\documentclass[12pt]{article}
\usepackage{amsmath}
\usepackage{tikz} 
\usetikzlibrary{arrows.meta}
\usepackage{tikz}
\usepackage{amsmath}
\usetikzlibrary{arrows.meta, bending, calc}
\usepackage{amssymb,amsfonts,amsmath}
\usepackage{float}
\usepackage{booktabs}
 \usepackage{adjustbox}
\usepackage{array}
\usepackage{graphicx}
\usepackage{multirow}
\usepackage{color}
\def\blue{\textcolor{blue}}
\def\red{\textcolor{red}}

\def\blue{\textcolor{blue}}
\usepackage{amsfonts, amsthm, amsmath}
\allowdisplaybreaks[4]
\usepackage{pgfplots}
\usepackage{xcolor}
\usepackage{afterpage}
\usepackage{rotating}

\usepackage{tikz}
\usepackage{lmodern}

\usetikzlibrary{decorations.pathmorphing}
\tikzset{
	redline/.style={red, thick, ->, shorten >=1pt, shorten <=1pt}
}

\usepackage{xcolor}
\usepackage{framed}
\usepackage{graphics}

\usepackage{amssymb}

\usepackage{amscd}

\usepackage[latin2]{inputenc}

\usepackage{t1enc}

\usepackage[mathscr]{eucal}

\usepackage{indentfirst}

\usepackage{enumitem}

\usepackage{enumitem}

\usepackage{graphicx}

\usepackage{graphics}

\usepackage{pict2e}

\usepackage{mathrsfs}

\usepackage{comment}

\usepackage{enumitem}

\usepackage{enumerate}
\usepackage[pagebackref]{hyperref}
\hypersetup{colorlinks=true}
\usepackage{color}
\usepackage{epic}
\usepackage{framed}
\usepackage{mathabx}
\usepackage{booktabs}
\usetikzlibrary{decorations.pathreplacing,calc}
\usepackage{makecell}
\newcolumntype{V}{!{\vrule width 2pt}}

\numberwithin{equation}{section}

\def\blue{\textcolor{blue}}
\def\red{\textcolor{red}}

\def\blue{\textcolor{blue}}

\theoremstyle{plain}

\newtheorem{theorem}{Theorem}[section]

\newtheorem{conjecture}[theorem]{Conjecture}
\newtheorem{remark}[theorem]{Remark}
\newtheorem{lemma}[theorem]{Lemma}
\newtheorem{definition}[theorem]{Definition}

\definecolor{handred}{RGB}{220, 0, 0}
\definecolor{handgreen}{RGB}{0, 150, 0}

\def\des{\mathsf{des}}
\def\maj{\mathsf{maj}}

\def\exc{\mathsf{exc}}
\def\den{\mathsf{den}}
\def\st{\mathsf{st}}
\def\Inv{\mathsf{Inv}}
\def\Imv{\mathsf{Imv}}
\def\Des{\mathsf{Des}}
\def\Exc{\mathsf{Exc}}
\def\Nexc{\mathsf{Nexc}}
\def\Nexcp{\mathsf{Nexcp}}
\def\Excp{\mathsf{Excp}}

\def\inv{\mathsf{inv}}

\def\imv{\mathsf{imv}}
\def\green{\textcolor{blue}}

\begin{document}
	\begin{center}
		{\Large\bf Further refinements of  Euler-Mahonian statistics for multipermutations}
	\end{center}
	
	\begin{center}
		
		{\small Kaimei Huang,  Yongzhou Wen,
			Sherry H.F. Yan$^{*}$\footnote{$^*$Corresponding author. \\{\em E-mail address:}   hfy@zjnu.cn (S.H.F. Yan). }} 
		
		 Department of Mathematics,
		Zhejiang Normal University\\
		Jinhua 321004, P.R. China

	\end{center}

	\noindent {\bf Abstract.}
	Permutation statistics constitute a classical subject of enumerative combinatorics.  In her study of the genus zeta function,  Denert  discovered a new Mahonian statistic for permutations, which is called the Denert's statistic  ({\bf $\den$}) by Foata and Zeilberger. As natural extensions of the $r$-descent number ({\bf $r\des$}) and the $r$-major index ({\bf $r\maj$}) introduced by Rawlings,  Liu introduced    the $g$-gap $\ell$-level descent number  ({\bf $g\des_{\ell}$})  and the $g$-gap $\ell$-level  major index   ({\bf $g\maj_{\ell}$}) for permutations. 
	In this paper, we introduce  the $g$-gap $\ell$-level Denert's statistic ({\bf $g\den_{\ell}$})  and the $g$-gap $\ell$-level excedance number ({\bf $g\exc_{\ell}$}) for multipermutations, which serve as natural generalizations of the Denert's statistic ({\bf $\den$}) and the excedance number ({\bf $\exc$}) for multipermutations first introduced by Han.   By constructing   two  explicit bijections, we      establish the equidistribution of   the pairs $(g\exc_{\ell}, g\den_{h} )$ and $(g\des_{\ell}, g\maj_{\ell})$  over  multipermutations for all   $1\leq h\leq g+\ell$. Our result  provides a new proof of the equidistribution of the pairs    ($\des$, $\maj$) and ($\exc$, $\den$) over multipermutations originally  derived by Han and  enables us to confirm a recent conjecture posed by Huang-Lin-Yan.   Furthermore, we  demonstrate  that for
	all  $1\leq h\leq g+\ell$,  the pair  $(g\exc_\ell, g\den_{h})$ is $r$-Euler-Mahonian over   multipermutations of $M=\{1^k, 2^k, \ldots, n^k\}$ where $r=g+\ell-1$ and $k\geq 1$, which extends a recent novel result derived by Liu from permutations to multipermutations.


	\section{Introduction}
	Let $M=\{1^{k_1}, 2^{k_2}, \ldots, n^{k_n}\}$  be a multiset where $i^{k_i}$ denotes   $k_i$ occurrences of $i$.  Unless specified otherwise, we always assume that $M=\{1^{k_1}, 2^{k_2}, \ldots, n^{k_n}\}$ where $k_i\geq 1$ and $k_1+k_2+\cdots+k_n=m$. Let $\mathfrak{S}_{M}$ denote the set of all multipermutations  on $M$. For a positive integer $n$, let $[n]:=\{1, 2, \ldots, n\}$. When $M=[n]$, we simply write $\mathfrak{S}_{M}$ as $\mathfrak{S}_{n}$.

	Given a multipermutation  $w=\alpha_1 \alpha_2 \ldots \alpha_m\in \mathfrak{S}_{M}$, 
	the  {\em descent set}  of $w$ is defined   to be
	$$
	\Des(w)=\{i\in[m-1]\mid \alpha_i > \alpha_{i+1}\},
	$$
	and 
the  {\em descent number} of $w$, denoted by $\des(w)$,  is defined to be	the cardinality of its descent set. 
 The {\em major index} of $w$, denoted by $\maj(w)$, is defined to be
	$$
	\maj(w)=\sum\limits_{i\in \Des(w) } i.
	$$
	A pair $(i,j)$ is called an {\em inversion (resp.,~weakly inversion) pair} of $w$ if $i<j$ and $\alpha_i>\alpha_j$ (resp.,~$\alpha_i\geq\alpha_j$).  Let $\Inv(w)$ (resp.,~$\Imv(w)$) denote the set of all inversion (resp.,~weakly inversion) pairs of $w$. The  {\em inversion (resp.,~weakly inversion) number} of $w$ is defined as $\inv(w)=|\Inv(w)|$ (resp.,~$\imv(w)=|\Imv(w)|$).    For example, if  $w=4121155325$, then  $\Des(w)=\{1,3,7,8\}$, $\des(w)=4$, $\maj(w)=1+3+7+8=19$, $\inv(w)=13$ and $\imv(w)=20$. MacMahon's equidistribution theorem \cite{Mac} asserts  that  
	\begin{equation}\label{invmaj}
		\sum\limits_{w\in \mathfrak{S}_M}  q^{\inv(w)}=\sum\limits_{w\in \mathfrak{S}_M}q^{\maj(w)}. 
	\end{equation}
 A  statistic  is said to be  {\em Mahonian} if it is  equidistributed  with $\inv$ over $\mathfrak{S}_M$. 
	 MacMahon  type results for  other combinatorial objects constitute a classical subject of enumerative combinatorics and   have been investigated extensively  in the literature;  
	 see \cite{Chen1, Chen2, Liu1, Liu2, RW, Wilson, Yan} and the references therein. 
	
	As a natural generalization of  the descent number and the major index, 
	Liu~\cite{Liujcta} introduced the notions of the  $g$-gap $\ell$-level descent number  $g\des_\ell$ and  the $g$-gap $\ell$-level major index $g\maj_\ell$ for permutations.  Here we extend Liu's definitions from permutations to multipermutations. 
	Let $g,\ell\geq 1$. Given a multipermutation $w=\alpha_1\alpha_2\ldots \alpha_m\in \mathfrak{S}_{M}$, let 
	$$
	g\Des_\ell(w)=\{i\in [n-1]\mid \alpha_{i} \geq  \alpha_{i+1}+g, \alpha_{i+1} \geq \ell\},
	$$
	$$
	g\Inv_\ell(w)=\{(i,j)\in \Inv(w)\mid\alpha_{j}<\ell \,\, \mbox{or}\,\, \alpha_i<\alpha_j+g\}.
	$$
	An index $i$ is said to be an {\em $g$-gap $\ell$-level descent}  of $w$ if $i\in g\Des_\ell(w) $. The {\em $g$-gap $\ell$-level descent number} of $w$, denoted by $g\des_\ell(w)$, is defined to be the number of $g$-gap $\ell$-level descents of $w$.
	The {\em $g$-gap $\ell$-level major index} of $w$ is defined to be 
	\begin{equation*}
		g\maj_\ell(w) = \sum\limits_{i\in g\Des_\ell(w)}i +|g\Inv_\ell(w)|. 
	\end{equation*}
	For example, if we let $w=4121155325$, then  $2\Des_3(w)=\{7\}$, $2\des_3(w)=1$ and $2\Inv_3=\{(1,2),(1,3),(1,4),(1,5),(1,8),(1,9),(3,4),(3,5),(6,9),(7,9),(8,9)\}$. Thus $2\maj_3(w)$ is given by 
	$$
	\begin{array}{ll}
		2\maj_3(w)
		&=\sum\limits_{i\in 2\Des_3(w)}i +|2\Inv_3(w)|\\
		&=7+11\\
		&=18.
	\end{array}
	$$
	When $g=\ell=1$, the statistics $g\maj_{\ell}$ and $g\des_{\ell}$ reduce to the   major index   $\maj$  and the   descent number $\des$, respectively.
	The $r$-gap $1$-level major index  of $w$  is also referred as  {\em     $r$-major index}  of $w$, denoted by $r\maj(w)$,  and the $r$-gap $1$-level descent number  of $w$ is called the {\em $r$-descent number}  of $w$, denoted by $r\des(\pi)$. These  notations were first  introduced   by Rawlings \cite{Raw, Raw1}  who also established the equidistribution of   the statistics $r\maj$ and $\inv$   over $\mathfrak{S}_M$.  
	The $1$-gap $r$-level major index  of $w$ is called the {\em     $r$-level major index}  of $w$, denoted by $ \maj_r(w)$,  and the $1$-gap $r$-level descent  of $w$ is called the {\em $r$-level descent number}  of $w$, denoted by $ \des_{r}(w)$ in \cite{Liujcta}.
	\begin{definition}
		A pair  $(\st_1, \st_2)$  of permutation statistics is said to be {\bf $r$-Euler-Mahonian} over  $\mathfrak{S}_{M}$ if 
		$$
		\sum\limits_{w\in \mathfrak{S}_{M}}  t^{\st_1(w)}q^{\st_2(w)}=  \sum\limits_{w\in \mathfrak{S}_{M}}  t^{r\des(w)}q^{r\maj(w)}.
		$$
	\end{definition}
	The $1$-Euler-Mahonian statistic  is also referred as the {\em Euler-Mahonian} statistic. 
	Liu~\cite{Liujcta}  proved that  for all $g,\ell\geq 1$,  the pair $(g\des_{\ell}, g\maj_{\ell})$ is $r$-Euler-Mahonian
over permutations, that is, 	\begin{equation}\label{eq-r-Euler}
		\sum_{\pi\in \mathfrak{S}_n} t^{r\des(\pi)}q^{r\maj(\pi)} =\sum_{\pi\in \mathfrak{S}_n} t^{g\des_\ell(\pi)}q^{g\maj_{\ell}(\pi)},  
	\end{equation}
	where $r=g+\ell-1$.

	In the course of seeking  Euler-Mahonian statistics for multipermutations, Han \cite{Han1} extended  the {\em Denert's statistic}, originally  introduced by Denert \cite{Denert}  for  permutations in  her study of the genus zeta function,   to multipermutations.  
	Suppose that    $\bar{w}=x_1 x_2 \ldots x_m$ is the weakly increasing rearrangement of $w$.  An index $i$ with  $1\leq i < m$  is called an {\em excedance place}
	of $w$ if $\alpha_i > x_i$. Let $\Excp(w)$  denote the set of excedance places of $w$ and let $\exc(w)=|\Excp(w)|$ be the {\em excedance number} of $w$.  For instance, given   
	\begin{equation}\label{den}
	\begin{pmatrix}
		\bar{w}\\
		w
	\end{pmatrix}
	=
	\begin{pmatrix}
		\blue{\bf1}1\blue{\bf1}22\blue{\bf3}\blue{\bf4}555\\
		\blue{\bf4}1\blue{\bf2}11\blue{\bf5}\blue{\bf5}325 
	\end{pmatrix}, 
	\end{equation}
	 we have $\Excp(w)=\{1,3,6,7\}$  and $\exc(w)=4$.  Define $\Nexcp(w)=[m]\setminus\Excp(w)$ as   the set of non-excedance places of $w$. For $i\in[m]$, if $i\in\Excp(w)$, then $\alpha_i$ is called an {\em excedance letter} of $w$; otherwise, $i\in\Nexcp(w)$ and $\alpha_i$ is a {\em non-excedance letter} of $w$. Let $\Exc(w)$ and   $\Nexc(w)$ be the subsequences of $w$ formed by all the  excedance  and non-excedance letters of $w$, respectively. Then the {\em Denert's statistic} of $w$, denoted by $\den(w)$, is defined as
	$$
	\den(w)=\sum\limits_{i\in \Excp(w)}i+{\imv}(\Exc(w))+{\inv}(\Nexc(w)).
	$$
	Continuing with the example $w=4121155325$ in (\ref{den}), we have
	$\Exc(w)=4255$  and $\Nexc(w)=111325$.
	Then $\den(w)$ is given by 
	$$
	\begin{array}{ll}
		\den(w)
		&=\sum\limits_{i\in \Excp(w)}i+{\imv}(\Exc(w))+{\inv}(\Nexc(w))\\
		&=1+3+6+7+{\imv}(4255)+{\inv}(111325)\\
		&=17+2+1\\
		&=20.
	\end{array}
	$$

  Denert~\cite{Denert} proposed an intersting   conjecture which asserts that 
	$ (\exc, \den)$ is Euler-Mahonian on permutations, that is, 
	\begin{equation}\label{excden}
		\sum\limits_{\pi\in \mathfrak{S}_{n}}  t^{\exc(\pi)}q^{\den(\pi)}=  \sum\limits_{\pi\in \mathfrak{S}_{n}}  t^{\des(\pi)}q^{\maj(\pi)}.
	\end{equation}
	This conjecture was first proved by
	Foata and Zeilberger \cite{Foata-Zeilberger}, and    a bijective proof was later provided by Han \cite{Han}.   The following multipermutation version of~\eqref{excden}  was attributed to Han~\cite{Han1}:
	\begin{equation}\label{den-majword}
		\sum\limits_{w\in \mathfrak{S}_{M}}  t^{\des(w)}q^{\maj(w)} =  
		\sum\limits_{w\in \mathfrak{S}_M}  t^{\exc(w)}q^{\den(w)}.
	\end{equation}

	Let  $w=\alpha_1 \alpha_2 \ldots \alpha_m\in \mathfrak{S}_M$ be a multipermutation whose weakly increasing  rearrangement is $\bar{w}=x_1 x_2 \ldots x_m$. An index $i\in[m]$ is called an {\em $g$-gap $\ell$-level  excedance place}  and $\alpha_i$ is called an {\em $g$-gap $\ell$-level  excedance letter} if $\alpha_i\geq x_i+g$ and $\alpha_i\geq \ell$.   We denote by $g\Excp_\ell(w)$ the set of $g$-gap $\ell$-level excedance places of $w$ and by   $g\Nexcp_\ell(w)$   the set of non-$g$-gap-$\ell$-level excedance places of $w$, respectively.  
	Assume that  $g\Excp_\ell(w)=\{i_1, i_2, \ldots, i_k\}$ with $i_1<i_2<\cdots<i_k$ and  $g\Nexcp_\ell(w)=\{j_1, j_2, \ldots, j_{m-k}\}$
	with $j_1<j_2<\cdots<j_{m-k}$.  Define
	$ 
	g\Exc_\ell(w)=\alpha_{i_1}\alpha_{i_2}\ldots \alpha_{i_k},
	$ 
	and 
	$ 
	g\Nexc_\ell(w)=\alpha_{j_1}\alpha_{j_2}\ldots \alpha_{j_{m-k}}.
	$ 
		\begin{definition}
		The {\bf $g$-gap $\ell$-level  Denert's statistic} of $w$, denoted by $g\den_{\ell}(w)$, is defined as
		$$
		g\den_{\ell}(w)=\sum\limits_{i\in g\Excp_{\ell}(w)} (i+B^{g}_i(w)) +{\imv}(g\Exc_{\ell}(w))+{\inv}(g\Nexc_{\ell} (w)), 
		$$
		where $B^{g}_i(w)=|\{j \mid \alpha_i -g< x_j < \alpha_i\}|$.
	\end{definition}
For example, if  
\begin{equation}\label{eg:gdenl}
	\begin{pmatrix}
		\bar{w}\\
		w
	\end{pmatrix}
	=
	\begin{pmatrix}
		\blue{\bf1}1122\blue{\bf3}\blue{\bf4}556\\
		\blue{\bf4}2311\blue{\bf5}\blue{\bf6}125 
	\end{pmatrix},
\end{equation}	
then we have $2\Excp_4(w)=\{1,6,7\}$, $2\Exc_4(w)=456$ and $2\Nexc_4(w)=2311125$. Hence, $2\den_4(w)$ is given by
$$
\begin{array}{ll}
	2\den_4(w)&=\sum\limits_{i\in 2\Excp_4(w)} (i+B^{2}_i(w)) +{\imv}(2\Exc_4(w))+{\inv}(2\Nexc_4 (w))\\
	&=(1+1)+(6+1)+(7+2)+{\imv}(456)+{\inv}(2311125)\\
	&=18+0+7\\
	&=25.
\end{array}
$$	
	\begin{definition}
		The {\bf  $g$-gap $\ell$-level excedance number} of $w$,  denoted by $g\exc_\ell(w)$,  is defined as
		$$g\exc_\ell(w)=|\{i\mid \alpha_i\geq x_i+g, x_i\geq \ell\}|.$$ 
	\end{definition}
Continuing with the example $w=4231156125$ in~\eqref{eg:gdenl}, we have $2\exc_4(w)=1$.
Clearly, $g\den_\ell$ and $g\exc_\ell$ reduces to  $\den$ and $\exc$ respectively  when $g=\ell=1$.

\begin{remark}
It is easily seen that 	for any $\pi \in \mathfrak{S}_n$,   we have
	$$
	g\den_{\ell}(\pi)=\sum\limits_{i\in g\Excp_{\ell}(\pi)} (i+g-1) +{\inv}(g\Exc_{\ell}(\pi))+{\inv}(g\Nexc_{\ell} (\pi)), 
	$$
	and 
	$$
	 g\exc_\ell(\pi)=|\{i\mid \pi_i\geq i+g, i\geq \ell\}|,
	$$
	 which coincide with the statistics $g\den_{\ell}$ and $g\exc_{\ell}$ employed  in \cite{Huang-yan AAM, Liujcta}. 
	\end{remark}

	The $r$-gap $1$-level Denert's statistic    of $w$ is called the {\em     $r$-gap Denert's statistic }  of $w$, denoted by $r\den(w)$,  and the $r$-gap $1$-level excedance number of $w$ is called the {\em $r$-gap excedance number}  of $w$, denoted by $r\exc(w)$.   Similarly,  The $1$-gap $r$-level Denert's statistic    of $w$ is called the {\em     $r$-level Denert's statistic }  of $w$, denoted by $ \den_r(w)$,  and the $1$-gap $r$-level excedance  number of $w$ is called the {\em $r$-level excedance number}  of $w$, denoted by $ \exc_r(w)$.

Motivated by finding refinments  of (\ref{excden}), Liu \cite{Liujcta}  proved that
\begin{equation}\label{liu-rden}
	\sum_{\pi\in \mathfrak{S}_n}t^{r\des(\pi)}q^{r\maj(\pi)}=\sum_{\pi\mathfrak{S}_n}t^{r\exc(\pi)} q^{r\den(\pi)}
\end{equation}
and further conjectured that 
\begin{equation}\label{HuangJCTA}
	\sum_{\pi\in \mathfrak{S}_n}t^{r\des(\pi)}q^{r\maj(\pi)}=\sum_{\pi\mathfrak{S}_n}t^{ \exc_r(\pi)} q^{\den_r(\pi)} .
\end{equation}
This conjecture  was later confirmed by Huang-Lin-Yan \cite{Huang-lin-yan jcta}. 
As  a further  extension of (\ref{liu-rden}) and  (\ref{HuangJCTA}), Huang-Yan \cite{Huang-yan AAM} proved that for all $g,\ell\geq 1$,   
\begin{equation}\label{HuangAAM}
	\sum_{\pi\in \mathfrak{S}_n}t^{r\des(\pi)}q^{r\maj(\pi)}=\sum_{\pi\mathfrak{S}_n}t^{ g\exc_\ell(\pi)} q^{g\den_\ell(\pi)}=\sum_{\pi\mathfrak{S}_n}t^{ g\exc_\ell(\pi)} q^{g\den_{g+\ell}(\pi)},
\end{equation}
where $r=g+\ell-1$.   It should be noted  that the first equality in (\ref{HuangAAM}) was originally conjectured by Liu \cite{Liujcta}. 
 Combining (\ref{eq-r-Euler}) and (\ref{HuangJCTA}), we derive that 
\begin{equation}\label{eq-des}
	\sum_{\pi\in \mathfrak{S}_n}t^{ \des_r(\pi)}q^{ \maj_r(\pi)}=\sum_{\pi\mathfrak{S}_n}t^{  \exc_r(\pi)} q^{ \den_r(\pi)}. 
\end{equation}
 Very recently, Liu \cite{Liu3} provided  a new bijective proof of the equidistribution of $(\des, \maj)$ and $(\exc, \den)$ over $\mathfrak{S}_n$. As an application, he also derived the following generalization of (\ref{HuangAAM}).
\begin{theorem}[Liu~\cite{Liu3}]\label{th-Liu3}
	For all $g, \ell\geq 1$ and $1\leq h\leq g+\ell$, 
	the pair $(g\exc_{\ell}, g\den_{h})$ is $r$-Euler-Mahonian over $\mathfrak{S}_n$, that is, 
	\begin{equation}\label{eq-liu2}
		\sum_{\pi\in\mathfrak{S}_n} t^{r\des(\pi)} q^{r\maj(\pi)}
		=\sum_{\pi\in \mathfrak{S}_n} t^{g\exc_\ell(\pi)}  q^{g\den_h(\pi)},
	\end{equation}
	where $r=g+\ell-1$. 
\end{theorem}
In views of (\ref{eq-r-Euler}) and (\ref{eq-liu2}), we derive that
\begin{equation}\label{eq-liu3}
	\sum_{\pi\in\mathfrak{S}_n} t^{g\des_{\ell}(\pi)} q^{g\maj_\ell(\pi)}
	=\sum_{\pi\in \mathfrak{S}_n} t^{g\exc_\ell(\pi)}  q^{g\den_h(\pi)}.
\end{equation}
for all $1\leq h\leq g+\ell$. 

In analogy to (\ref{liu-rden}) derived by Liu \cite{Liujcta}, 	 Huang-Lin-Yan \cite{Huang-lin-yan2} proved that 	for all $r\geq 1$,   the pair $(r\exc, r\den)$ is $r$-Euler-Mahonian over $\mathfrak{S}_M$, that is,  
	\begin{equation}\label{eq-Huang-JCTA2}
		\sum_{w\in\mathfrak{S}_M} t^{r\des(w)} q^{r\maj(w)}
		=\sum_{w\in \mathfrak{S}_M} t^{r\exc(w)}  q^{r\den(w)}.			
	\end{equation}	
	This is a   natural refinement of ~\eqref{den-majword} derived by Han~\cite{Han1}.  
In oder to 	generalize   (\ref{eq-des}) to multipermutations,   Huang-Lin-Yan \cite{Huang-lin-yan2}  further proposed the following  conjecture.   
	\begin{conjecture}[Huang--Lin--Yan~\cite{Huang-lin-yan2}]\label{con-rden-word}
		For all $r\geq 1$, we have 
		\begin{equation}
			\sum_{w\in\mathfrak{S}_M} t^{\des_r(w)} q^{\maj_r(w)}
			=\sum_{w\in \mathfrak{S}_M} t^{\exc_r(w)}  q^{\den_r(w)}.			
		\end{equation}	
	\end{conjecture}

 In this paper, we shall  first  establish   two  bijections between $\mathfrak{S}_{M'}\times \{0, 1,\ldots, m'\}$ and $\mathfrak{S}_{M}$, where $M'=M\setminus \{n^{k_n}\}$  and $m'=m-k_n$. Our bijections  allow   us to 
 derive a common   generalization of     of (\ref{eq-liu3})  as follows. 
 
 	\begin{theorem}\label{thm-den-r-h}
 	For all $g,\ell\geq 1$ and $1\leq h\leq g+\ell$, we have 
 	\begin{equation}\label{eq-den-r-h}
 		\sum_{w\in\mathfrak{S}_M} t^{g\des_\ell(w)} q^{g\maj_\ell(w)}
 		=\sum_{w\in \mathfrak{S}_M} t^{g\exc_\ell(w)}  q^{g\den_h(w)}.			
 	\end{equation}	
 \end{theorem} 
From Table \ref{table1}, we can easily  verify  that 
$$
\begin{array}{lll}
	\sum\limits_{w\in \mathfrak{S}_{\{{1,2^2,3^2}\}}}  t^{\des_2(w)}q^{\maj_2(w)}
	&=(1+2q+3q^2+2q^3+q^4)+t(q^2+3q^3+5q^4+5q^5+3q^6+q^7)\\
	&+t^2(q^6+q^7+q^8)\\
	&&\\
	&=\sum\limits_{w\in \mathfrak{S}_{\{{1,2^2,3^2}\}}}  t^{\exc_2(w)}q^{\den_h(w)} 
\end{array}
$$
for all $1\leq h\leq 3$.

Setting $g=1$ and $h=\ell=r$ in Theorem \ref{thm-den-r-h} yields  a proof of Conjecture \ref{con-rden-word}.  Putting   $g=\ell=h=1$ in Theorem \ref{thm-den-r-h}  provides  a new proof of the equidistribution of  $(\des, \maj)$ and $(\exc, \den)$ originally derived by Han \cite{Han1}.

\begin{table}[htb]\label{table1}
	\setlength\tabcolsep{6pt}   
	\begin{center}
		\caption{The values of the pairs ($2\des$, $2\maj$), ($\des_2$, $\maj_2$), ($\exc_2$, $\den_2$), ($\exc_2$, $\den$) and ($\exc_2$, $\den_3$) of each multipermutaion in $\mathfrak{S}_{\{1,2^2,3^2\}}$.}
		\vspace{1em}
		\fontsize{10}{8}\selectfont
		\begin{tabular}{cccccccc}
			\toprule
			\boldmath{$\mathfrak{S}_{\{1, 2^2, 3^2\}}$}& 
			\boldmath{($2\des$, $2\maj$)}&
			\boldmath{($\des_2$, $\maj_2$)}&
			\boldmath{($\exc_2$, $\den_2$)}&
			\boldmath{($\exc_2$, $\den$)}& \boldmath{($\exc_2$, $\den_3$)} \\
			\midrule
			12233 & (0,0) & (0,0) & (0,0) & (0,0) & (0,0)\\
			12323 & (0,1) & (1,3) & (1,3) & (1,3) & (1,3)\\
			12332 &	(0,2) & (1,4) &	(1,4) & (1,4) & (1,4)\\
			13223 &	(0,2) & (1,2) &	(1,2) & (1,2) & (1,2)\\
			13232 &	(0,3) & (2,6) &	(1,3) & (1,3) & (1,3)\\
			13322 &	(0,4) & (1,3) &	(2,6) & (2,6) & (2,6)\\
			21233 &	(0,1) & (0,1) &	(0,1) & (0,1) & (0,1)\\
			21323 &	(0,2) & (1,4) &	(1,4) & (1,4) & (1,4)\\
			21332 &	(0,3) & (1,5) &	(1,5) & (1,5) & (1,5)\\
			22133 &	(0,2) & (0,2) &	(0,2) & (0,2) & (0,2)\\
			22313 &	(1,5) & (0,3) &	(1,5) & (1,5) & (1,5)\\
			22331 &	(1,6) & (0,4) &	(1,6) & (1,6) & (1,6)\\
			23123 &	(1,4) & (0,2) &	(1,3) & (1,3) & (1,3)\\
			23132 &	(1,5) & (1,6) &	(1,4) & (1,4) & (1,4)\\
			23213 &	(0,3) & (1,5) &	(1,4) & (1,4) & (1,4)\\
			23231 &	(1,7) & (1,6) &	(1,5) & (1,5) & (1,5)\\
			23312 &	(1,6) & (0,3) &	(2,7) & (2,7) & (2,7)\\
			23321 &	(0,4) & (1,7) &	(2,8) & (2,8) & (2,8)\\
			31223 &	(1,3) & (0,1) &	(0,1) & (0,1) & (0,1)\\
			31232 &	(1,4) & (1,5) &	(0,2) & (0,2) & (0,2)\\
			31322 & (1,5) & (1,4) & (1,5) & (1,5) & (1,5)\\
			32123 & (0,3) & (1,3) & (0,2) & (0,2) & (0,2)\\
			32132 & (0,4) & (2,7) & (0,3) & (0,3) & (0,3)\\
			32213 & (0,4) & (1,4) & (0,3) & (0,3) & (0,3)\\
			32231 & (1,8) & (1,5) & (0,4) & (0,4) & (0,4)\\
			32312 & (1,7) & (1,4) & (1,6) & (1,6) & (1,6)\\
			32321 & (0,5) & (2,8) & (1,7) & (1,7) & (1,7)\\
			33122 & (1,6) & (0,2) & (1,4) & (1,4) & (1,4)\\
			33212 & (0,5) & (1,5) & (1,5) & (1,5) & (1,5)\\
			33221 & (0,6) & (1,6) & (1,6) & (1,6) & (1,6)\\
			\bottomrule
		\end{tabular}
	\end{center}
\end{table}

Relying on Theorem \ref{thm-den-r-h}, we derive 
  the following interesting extension of Theorem \ref{th-Liu3} derived by Liu \cite{Liu3} to the multipermutations over $M=\{1^k, 2^k, \ldots, n^k\}$ with $k\geq 1$. 
	\begin{theorem}\label{th-regular}
		Let $M=\{1^{k}, 2^{k}, \ldots, n^{k}\}$ with $k\geq 1$. 
		For all $g, \ell\geq 1$ and $1\leq h\leq g+\ell$,  the pair  $(g\exc_\ell, g\den_{h})$ is $r$-Euler-Mahonian over $\mathfrak{S}_M$, that is 
		\begin{equation}\label{eq-regular}
		\sum_{w\in \mathfrak{S}_{M}}t^{r\des(w)}q^{r\maj(w)}=	\sum_{w\in \mathfrak{S}_{M}}t^{g\exc_\ell(w)}q^{g\den_{h}(w)},
			\end{equation} 
		where $r=g+\ell-1$.  
	\end{theorem}
\begin{remark}
 It should be noted  that Liu's bijection \cite{Liu3}  is not applicable to  multipermutations. 
  Moreover, (\ref{eq-regular})  does not hold for  arbitrary multisets.   For example, by Table \ref{table1}, one can easily check that 
   $$
  \begin{array}{lll}
  	\sum\limits_{w\in \mathfrak{S}_{\{{1,2^2,3^2}\}}}  t^{2\des(w)}q^{2\maj(w)}
  	&=(1+2q+4q^2+4q^3+4q^4+2q^5+q^6)\\
  	&+t(q^3+2q^4+3q^5+3q^6+2q^7+q^8)\\
  	&&\\
  	&\neq \sum\limits_{w\in \mathfrak{S}_{\{{1,2^2,3^2}\}}}  t^{\exc_2(w)}q^{\den_2(w)} 
  \end{array}
  $$
\end{remark}

	\section{Proofs} 
	We begin with some definitions and notations.  	For a nonnegative integer $n$,  the {\em q-integer} and the {\em q-factorial } is defined as
	$$
	[n]_q=1+q+q^2+\cdots +q^{n-1}\quad\text{and}\quad [n]_q!=[1]_q[2]_q\ldots [n]_q.
	$$
 Given  nonnegative  integers $t_1, t_2, \ldots, t_n$ with $t_1+t_2+\cdots+t_n=s$, 
the {\em$q$-multinomial coefficient} is defined as 
 $${ s\brack t_1, t_2, \ldots, t_n}={[s]_q!\over [t_1]_q![t_2]_q! \ldots [t_n]_q!}.$$  
 We abbreviate  ${n\brack k}$ for $n\brack k,n-k$, which is known as the {\em$q$-binomial coefficients} (or {\em Gaussian polynomials}).  
 
	A {\em partition} $\lambda$ of a positive integer $n$ is a finite non-increasing sequence of nonnegative integers  $(\lambda_1, \lambda_2, \ldots, \lambda_s)$ such that $\sum_{i=1}^s \lambda_i=n$.  The entries $\lambda_i$   are called the parts of $\lambda$.
	The weight of $\lambda$,  denote by $|\lambda|$,  is  
	the sum of its parts.   Let $\mathcal{P}(s,t)$ denote the set of partitions  with $s$ parts and whose largest part is at most $t$. 
	 The Gaussian polynomial ${s+t\brack s}$  has the following 
	 partition interpretation  (see \cite[Theorem 3.1]{Andrews}):
	 \begin{equation}\label{eqpar1}
	 	{s+t\brack s}=\sum_{\lambda\in \mathcal{P}(s,t) } q^{|\lambda|}.
	 \end{equation}

	Fix   $g, \ell\geq 1$. Let  $M'=M-\{n^{k_n}\}$ and $m'=m-k_n$.  For $n\geq \ell$, let 
	$$
	\delta_\ell:=k_1+k_2+\cdots+k_{\ell-1}
	$$
	the sum of consecutive $\ell-1$ terms of $k_i$'s starting from $k_1$ with the convention that $\delta_1=0$.   For $n\geq g$,  let
	$$
	\gamma_g:=k_{n-g+1}+k_{n-g+2}+\cdots+k_{n-1}
	$$
	the sum of consecutive $g-1$ terms of $k_i$'s starting from $k_{n-g+1}$ with the convention that $\gamma_1=0$.

To establish  the equidistribution   stated in Theorem \ref{thm-den-r-h},  we shall construct two bijections   between $\mathfrak{S}_{M'}\times \mathcal{P}(k_n, m')$ and  $ \mathfrak{S}_{M}$  and show   that these  bijections possess the following celebrated properties in the subsequent subsections. 

 \begin{theorem}\label{th-Phi-den}
 	For  all $g,    h\geq 1$ and  $n\geq \max\{g+1, h\}$,   there exists  a bijection $\Phi^{\den}_{g, h}$ between $\mathfrak{S}_{M'}\times \mathcal{P}(k_n, m')$ and  $ \mathfrak{S}_{M}$  such  that for any $(w, \lambda)\in   \mathfrak{S}_{M'}\times  \mathcal{P}(k_n, m')$,   we have 
 	\begin{equation}\label{eq-Phiden1}
 		g\den_h(\Phi^{\den}_{g, h}(w, \lambda))=g\den_{h}(w)+|\lambda|.
 	\end{equation}
 Furthermore,    if   $g\exc_{\ell}(w)=s$ and $g\exc_{\ell}(\Phi^{\mathrm{den}}_{g,h}(w, \lambda))=t$,   then   we have
 	\begin{equation}\label{eq-Phiden2}
 		\lambda_{t-s}\geq t+\delta_\ell+\gamma_g \geq \lambda_{t-s+1}
 	\end{equation} 
for all $\ell\geq 1$, $1\leq h\leq g+\ell $ and $n\geq g+\ell$.
 \end{theorem}

  \begin{theorem}\label{th-Phi-maj}
 	For all $g, \ell\geq 1$ and  $n\geq g+\ell$,   there exists  a bijection $\Phi^{\maj}_{g, \ell}$ between $\mathfrak{S}_{M'}\times \mathcal{P}(k_n, m')$ and  $ \mathfrak{S}_{M}$  such that for any $(w, \lambda)\in   \mathfrak{S}_{M'}\times  \mathcal{P}(k_n, m')$,   we have 
 	\begin{equation}\label{eq-Phi1}
 		g\maj_\ell(\Phi^{\maj}_{g, \ell}(w, \lambda))=g\maj_{\ell}(w)+|\lambda|.
 	\end{equation}
 	Furthermore,  if $g\des_{\ell}(w)=s$ and $g\des_{\ell}(\Phi^{\mathrm{maj}}_{g, \ell}(w, \lambda))=t$, then  we have
 	\begin{equation}\label{eq-Phi2}
 		\lambda_{t-s}\geq t+\delta_\ell+\gamma_g \geq \lambda_{t-s+1}.
 	\end{equation} 
 \end{theorem}
 MacMahon's equidistribution theorem \cite{Mac}  tells us that
 \begin{equation}\label{eq-invmaj}
 	\sum\limits_{w\in \mathfrak{S}_M}  q^{\inv(w)}=\sum\limits_{w\in \mathfrak{S}_M}q^{\maj(w)}={m\brack k_1,k_2,\ldots,k_n}. 
 \end{equation}
  In the following, we aim to show that $g\den_h$ and $\inv$ are equally distributed over $\mathfrak{S}_{M}$, and hence $g\den_{\ell}$ is Mahonian. 
  \begin{theorem}
  	For all $g,h\geq 1$, we have
  	\begin{equation}\label{eq-mah}
  		\sum_{w\in \mathfrak{S}_{M}}q^{g\den_h(w)}=\sum_{w\in \mathfrak{S}_{M}}q^{\inv(w)}={m\brack k_1,k_2,\ldots,k_n}.
  	\end{equation}
  \end{theorem}
\begin{proof}
	We proceed by induction on $n$. Observe that when $n<\max\{h, g+1\}$, 
 the statistic $g\den_{h}$ coincides with  the statistic $\inv$. This immediately implies that     (\ref{eq-mah}) is valid  for $ n<\max\{h, g+1\}$.       Assume that (\ref{eq-mah}) holds for $n-1$ with $n\geq \max\{h, g+1\} $, that is, 
	$$ \sum\limits_{w\in \mathfrak{S}_{M'}}q^{g\den_h(w)}=\sum\limits_{w\in \mathfrak{S}_{M'}}q^{\inv(w)}.$$
	Then, we deduce that
	$$
	\begin{array}{llll}
		\sum\limits_{w\in \mathfrak{S}_{M}}q^{g\den_h(w)}&=& \sum\limits_{w'\in \mathfrak{S}_{M'}}\sum\limits_{\lambda\in \mathcal{P}(k_n, m')}q^{g\den_{h}(w')+|\lambda|}&\,\,\, \,\,\mbox{(by Theorem \ref{th-Phi-den})}\\
		 [1.5em]
		&=&\big( \sum\limits_{w'\in \mathfrak{S}_{M'}}q^{g\den_{h}(w')}\big) \big(\sum\limits_{\lambda\in \mathcal{P}(k_n, m')}q^{ |\lambda|}\big)&\\
	 [1.5em]
		&=& {m\brack k_n}\sum\limits_{w'\in \mathfrak{S}_{M'}}q^{g\den_{h}(w')} & \,\,\, \,\,\mbox{(by  (\ref{eqpar1}))}\\
	  [1.5em]
		&=& {m\brack k_n} \sum\limits_{w'\in \mathfrak{S}_{M'}}q^{\inv(w')} & \,\,\, \,\,\mbox{(by  induction hypothesis)}\\
		  [1.5em]
			&=& {m\brack k_1,k_2,\ldots,k_n} & \,\,\, \,\,\mbox{(by  (\ref{eq-invmaj}))}\\
			 [1.5em]
				&=& \sum\limits_{w\in \mathfrak{S}_{M}}q^{\inv(w)}  & \,\,\, \,\,\mbox{(by  (\ref{eq-invmaj}))}\\
		\end{array}
	$$
	as desired, completing the proof. 
	\end{proof}

 \noindent{\bf Proof of Theorem \ref{thm-den-r-h}.}
 We proceed to prove the equidistribution of $(g\des_{\ell}, g\maj_{\ell})$ and $(g\exc_{\ell}, g\den_{h})$ by induction on $n$.   
 For  $n<g+\ell$,  it is obvious that    $g\des_\ell(w)=g\exc_\ell(w)=0$  and $g\maj_{\ell}(w)=\inv(w)$.   By (\ref{eq-mah}),  the statistics $g\den_{h}(w)$  and $\inv(w)$ are equally distributed over $\mathfrak{S}_{M}$.  This implies that  $g\den_{h}(w)$   and   $g\maj_{\ell}(w)$ have the same distribution over $\mathfrak{S}_{M}$ when $n<g+\ell$. Hence,  the pairs $(g\des_{\ell}, g\maj_{\ell})$ and $(g\exc_{\ell}, g\den_{h})$ are equally distributed over $\mathfrak{S}_{M}$ when $n<g+\ell$.   
 
 Now we assume that  the assertion also holds for  $n-1$ with $n\geq g+\ell$, that is, $(g\des_{\ell}, g\maj_{\ell})$ and $(g\exc_{\ell}, g\den_{h})$ are equally distributed over $\mathfrak{S}_{M'}$. 
  
    Let $\kappa=\delta_\ell+\gamma_g$. 
  Then,  we derive that
  $$
  \begin{array}{lll}
  	\sum\limits_{\substack{w\in \mathfrak{S}_{M}\\
  			g\exc_\ell(w)=t}}q^{g\den_h(w)}&=&\sum\limits_{s\leq t} \sum\limits_{\substack{w'\in \mathfrak{S}_{M'}\\
  			g\exc_\ell(w')=s}}\sum\limits_{\lambda_1\geq \lambda_2\geq \cdots \geq \lambda_{t-s}\geq t+\kappa \geq \lambda_{t-s+1}\geq \cdots\geq \lambda_{k_n}}q^{g\den_{h}(w')+|\lambda|} \\
  		 [2em]
  		&=&\sum\limits_{s\leq t} \big(\sum\limits_{\substack{w'\in \mathfrak{S}_{M'}\\
  				g\exc_\ell(w')=s}}q^{g\den_{h}(w')}\big)  \big(\sum\limits_{\lambda_1\geq \lambda_2\geq \cdots \geq \lambda_{t-s}\geq t+\kappa \geq \lambda_{t-s+1}\geq \cdots\geq \lambda_{k_n}}q^{ |\lambda|} \big)\\
  		 [2em]
  			&=&\sum\limits_{s\leq t} \big(\sum\limits_{\substack{w'\in \mathfrak{S}_{M'}\\
  				g\des_\ell(w')=s}}q^{g\maj_{\ell}(w')}\big)  \big(\sum\limits_{\lambda_1\geq \lambda_2\geq \cdots \geq \lambda_{t-s}\geq t+\kappa \geq \lambda_{t-s+1}\geq \cdots\geq \lambda_{k_n}}q^{ |\lambda|} \big)\\
  		 [2em]
  	&=&\sum\limits_{s\leq t} \sum\limits_{\substack{w'\in \mathfrak{S}_{M'}\\
  			g\des_\ell(w')=s}}\sum\limits_{\lambda_1\geq \lambda_2\geq \cdots \geq \lambda_{t-s}\geq t+\kappa \geq \lambda_{t-s+1}\geq \cdots\geq \lambda_{k_n}}q^{g\maj_{\ell}(w')+|\lambda|} \\
  		 [2em]
  		&=& \sum\limits_{\substack{w\in \mathfrak{S}_{M}\\
  				g\des_\ell(w)=t}}q^{g\maj_\ell(w)} 
  	\end{array}
  	$$
   as desired, where the first equality follows from Theorem \ref{th-Phi-den},  the third equality follows from the induction hypothesis, and the last equality follows from Theorem \ref{th-Phi-maj}.  This completes the proof.  \qed
 
 \noindent{\bf Proof of Theorem \ref{th-regular}.}
 Let $r=g+\ell-1$ and $M=\{1^k, 2^k, \ldots, n^k\}$.
 In view of Theorem \ref{thm-den-r-h}, to establish  (\ref{eq-regular}), it suffices to demonstrate that 
  $(r\des, r\maj)$ and $(g\des_\ell, g\maj_{\ell})$ are equally distributed over $\mathfrak{S}_{M}$.

  We  proceed   by induction on $n$.  For $n<g+\ell$, 
  it  is obvious that  
  $r\des(w)=g\des_{\ell}(w)=0$ and $r\maj(w)=g\maj_{\ell}(w)=\inv(w)$. This implies that  $(r\des, r\maj)$ and $(g\des_\ell, g\des_{\ell})$ are equally distributed over $\mathfrak{S}_{M}$ when $n<g+\ell$. 
  Now assume  that the assertion also holds for   $n-1$ with $n\geq g+\ell$, that is,  $(r\des, r\maj)$ and $(g\des_\ell, g\des_{\ell})$ are equally distributed over $\mathfrak{S}_{M'}$.  Let $\kappa=\delta_{\ell}+\gamma_{g}$.  It is straightforward to verify that  $\kappa=(g+\ell-2)k$.   Then,   we deduce that 
  \begin{equation}
  	\begin{array}{lll}
  		\sum\limits_{\substack{w\in \mathfrak{S}_{M}\\
  			g\des_\ell(w)=t}}q^{g\maj_\ell(w)}&=&\sum\limits_{s\leq t} \sum\limits_{\substack{w'\in \mathfrak{S}_{M'}\\
  			g\des_\ell(w')=s}}\sum\limits_{\lambda_1\geq \lambda_2\geq \cdots \geq \lambda_{t-s}\geq t+\kappa \geq \lambda_{t-s+1}\geq \cdots\geq \lambda_{k_n}}q^{g\maj_{\ell}(w')+|\lambda|}\\
  		 [2em]
  		&=&\sum\limits_{s\leq t} \sum\limits_{\substack{w'\in \mathfrak{S}_{M'}\\
  			r\des(w')=s}}\sum\limits_{\lambda_1\geq \lambda_2\geq \cdots \geq \lambda_{t-s}\geq t+(g+\ell-2)(k-1) \geq \lambda_{t-s+1}\geq \cdots\geq \lambda_{k_n}}q^{r\maj(w')+|\lambda|}\\
  		 [2em]
  				&=&\sum\limits_{s\leq t} \sum\limits_{\substack{w'\in \mathfrak{S}_{M'}\\
  					r\des(w')=s}}\sum\limits_{\lambda_1\geq \lambda_2\geq \cdots \geq \lambda_{t-s}\geq t+(r-1)(k-1) \geq \lambda_{t-s+1}\geq \cdots\geq \lambda_{k_n}}q^{r\maj(w')+|\lambda|}\\
  				 [2em]
  				&=&\sum\limits_{\substack{w\in \mathfrak{S}_{M}\\
  						r\des(w)=t}}q^{r\maj(w)},
  		\end{array}
  	 \end{equation}
  as desired, where the first and the last equalities follow from   Theorem \ref{th-Phi-maj}, and the second equality follows from the  induction hypothesis. This completes the proof. 
 \qed

 The rest of this section is organized as follows.   In Subsection 2.1,  in order to define the map $\Phi^{\den}_{g, h}$, we fisrt develop two crucal maps.  In Subsection 2.2, relying on the two maps defined in Subsection 2.1,  we construct the map $\Phi^{\den}_{g,h}$  and  complete  the proof of Theorem \ref{th-Phi-den}.  In Subsection 2.3, we establish the map $\Phi^{\maj}_{g, \ell}$ and prove that it  exhibits the properties as claimed in Theorem \ref{th-Phi-maj}. 
 \subsection{Introducing two crucial  maps}
 Before defining  the map $\Phi^{\den}_{g, h}$, we  first  develop  two crucial  maps that would play an essential role in the construction of the map $\Phi^{\den}_{g, h}$. 
Throughout  this subsection, we always assume that $g,h\geq 1$, $n\geq \max\{g+1, h\}$,    $N=M'\cup \{n^{a}\}$ and $N'=M'\cup \{n^{a-1}\}$ with $a\geq 1$.  
 Our goal is to construct  a map $\phi_{n,g,h}:\mathfrak{S}_{N'}\times  \{0,1,\ldots, m'\} \longrightarrow \mathfrak{S}_{N}$.  To this end, we  first introduce a   labeling scheme for  the spaces of  $w$. 
 \begin{definition}\label{def-den_r-label}
	Let $w=\alpha_1\alpha_2\ldots \alpha_{m'+a-1}\in \mathfrak{S}_{N'}$ and define  $\gamma_g=k_{n-g+1}+k_{n-g+2}+\cdots+k_{n-1}$.  Assume that $w$ has exactly $s$ $g$-gap $h$-level excedance places.  Then 
	the  {\bf  $g\den_{h}$-labeling} of $w$  is constructed as follows.
	\begin{itemize}
		\item  Star the space after each $\alpha_{i} $ for all $i>m'$.
		\item Label the  rightmost $\gamma_g+1$ unstarred spaces  by $0,1,\ldots, \gamma_g$. 
		\item Label the  spaces before the $g$-gap $h$-level excedance letters from right to left    with $\gamma_g+1,\gamma_g+2,\ldots, \gamma_g+s$. 
		\item  Label the remaining unstarred spaces  from left to right with $\gamma_g+s+1, \ldots, m'$.
	\end{itemize}	
\end{definition}

Take  $N=\{1^3, 2^3, 3, 4, 5^2, 6^2, 7^{2}\}$,    $w=5121264732165$ and $g=h=3$ for example.  Clearly,  we have  $\gamma_3=4$. Moreover, all the $3$-gap $3$-level excedance letters are given by $\alpha_1$, $\alpha_6$ and $\alpha_8$ which are in red. Then the $3\den_3$-labeling of $w$ is given by 
\begin{center}
\begin{table}[h]
	\centering
	\vspace{0.5em}
	\fontsize{12}{10}\selectfont
	\renewcommand{\arraystretch}{1.4}
	\setlength{\tabcolsep}{0.8em}
	\makebox[\textwidth][c]{
		\begin{tabular}{c|ccccccccccccccc}
			$i$ & 1 & 2 & 3 & 4 & 5 & 6 & 7 & 8 & 9 & 10 & 11 & 12 & 13 \\
			\hline
			\multirow{2}{*}{$\begin{array}{c} x_i \\ \alpha_i \end{array}$} 
			& \makebox[0.7em][r]{1} & \makebox[0.7em][r]{1} & \makebox[0.7em][r]{1} & \makebox[0.7em][r]{2} & \makebox[0.7em][r]{2} & \makebox[0.7em][r]{2} & \makebox[0.7em][r]{3} & \makebox[0.7em][r]{4} & \makebox[0.7em][r]{5} & \makebox[0.7em][r]{5} & \makebox[0.7em][r]{6} & \makebox[0.5em][r]{6} &
			\makebox[0.2em][r]{7}  \\
			& \makebox[0.7em][r]{$_{\blue{7}}\red{5}$} 
			& \makebox[0.7em][r]{$_{\blue{8}}1$} 
			& \makebox[0.7em][r]{$_{\blue{9}}2$} 
			& \makebox[0.7em][r]{$_{\blue{10}}1$} 
			& \makebox[0.7em][r]{$_{\blue{11}}2$} 
			& \makebox[0.7em][r]{$_{\blue{6}}\red{6}$} 
			& \makebox[0.7em][r]{$_{\blue{12}}4$} 
			& \makebox[0.7em][r]{$_{\blue{5}}\red{7}$} 
			& \makebox[0.7em][r]{$_{\blue{4}}3$} 
			& \makebox[0.7em][r]{$_{\blue{3}}2$}
			& \makebox[0.7em][r]{$_{\blue{2}}1$}
			& \makebox[1.4em][r]{$_{\blue{1}}6_{\blue{0}}$}
			& \makebox[1.1em][r]{$5_*$} \\
		\end{tabular}
	}
\end{table}
\end{center}
where the labels of the spaces are written as subscripts.

 \begin{framed}
	\begin{center}
		{\bf The map $\phi_{n,g, h}:\mathfrak{S}_{N'}\times  \{0,1,\ldots, m'\}\longrightarrow \mathfrak{S}_{N}$ }
	\end{center}
	Let $(w, c)\in  \mathfrak{S}_{N'}\times\{0, 1, \ldots, m'\} $ with $w=\alpha_1\alpha_2\ldots \alpha_{m'+a-1}$ and $\bar{w}=x_1x_2\ldots x_{m'+a-1}$.     Assume that the space before $\alpha_y$ is labeled by $c$ under the $g\den_h$-labeling of $w$ when $c>0$.  
	Define  $u=\phi_{n,g,h}(w, c)$ to be the multipermutation   constructed by distinguishing the following three cases. 
	\begin{itemize}
		\item Case 1: $c=0$. \\
	Construct   a multipermutation    $u$  from $w$ by inserting  an $n$  immediately after   $\alpha_{m'}$.

		\item Case 2:   $c>0$ and $y\in g\Excp_h(w)$.   \\
		Construct   a multipermutation    $u$  from $w$ by the following procedure.
		\begin{itemize}
	 \item  Find all the $g$-gap $h$-level excedance letters   of $w$  occurring      weakly to the right of   $\alpha_y$,  say
	 $\alpha_{i_1}, \alpha_{i_2}, \ldots, \alpha_{i_a}$ with
	  $ y=i_1<i_2<\cdots <i_a$. Find the smallest  integer $k$ satisfying $\alpha_{i_k}< x_{i_{k+1}}+g$   with the convention that $x_{i_{a+1}}=n$.
		\item    
	Choose the smallest integer  $p$ such that $x_p=\alpha_{i_k}-g+1$.	Find all the  non-$g$-gap-$h$-level excedance letters of $w$ that occur   weakly   to  the right  of $\alpha_{p}$, say $\alpha_{j_1}, \alpha_{j_2}, \ldots, \alpha_{j_b}$ with $ j_1<j_2<\cdots<j_b$.    
			\item   Replace  $\alpha_{i_1}=\alpha_y$ with an  $n$ and replace   $\alpha_{i_{z}}$  with $\alpha_{i_{z-1}}$  for all $1< z\leq k$. 
			
			\item  Replace  $\alpha_{j_1}$ with $\alpha_{i_k}$,  and  replace  $\alpha_{j_{z+1}}$ with $\alpha_{j_{z}}$ for all $1\leq z\leq b $  with the convention that $j_{b+1}=m'+a$.   
		 \end{itemize}
		 
		 	\item Case 3:   $c>0$ and  $y\notin g\Excp_h(w)$. \\ 
		 Find all the  non-$g$-gap-$h$-level excedance letters of $w$ that occur    to  the right  of $\alpha_{y}$ , say $\alpha_{j_1}, \alpha_{j_2}, \ldots, \alpha_{j_b}$ with $ j_1<j_2<\cdots<j_b$.   Generate  a multipermutation $u$ from $w$ by  replacing $\alpha_{y}$ with an  $n$,   and replacing $\alpha_{j_{z}}$ with $\alpha_{j_{z-1}}$ for all $1\leq z\leq b+1 $ with the convention that $j_{b+1}=m'+a$ and $j_0=y$.
		 
	\end{itemize}

\end{framed}

Take $N=\{1^3, 2^3, 3, 4, 5^2, 6^2, 7^2\}$ and $g=h=3$ for example. Let  
$w=5121264732165$ and $c=7$. 
Then the $3\den_3$-labeling of $w$ is given by 
\begin{center}
\begin{table}[h]
	\centering
	\vspace{0.5em}
	\fontsize{12}{10}\selectfont
	\renewcommand{\arraystretch}{1.4}
	\setlength{\tabcolsep}{0.8em}
	\makebox[\textwidth][c]{
		\begin{tabular}{c|ccccccccccccccc}
			$i$ & 1 & 2 & 3 & 4 & 5 & 6 & 7 & 8 & 9 & 10 & 11 & 12 & 13 \\
			\hline
			\multirow{2}{*}{$\begin{array}{c} x_i \\ \alpha_i \end{array}$} 
			& \makebox[0.7em][r]{1} & \makebox[0.7em][r]{1} & \makebox[0.7em][r]{1} & \makebox[0.7em][r]{2} & \makebox[0.7em][r]{2} & \makebox[0.7em][r]{2} & \makebox[0.7em][r]{3} & \makebox[0.7em][r]{4} & \makebox[0.7em][r]{5} & \makebox[0.7em][r]{5} & \makebox[0.7em][r]{6} & \makebox[0.5em][r]{6} &
			\makebox[0.2em][r]{7}  \\
			& \makebox[0.7em][r]{$_{\blue{7}}\red{5}$} 
			& \makebox[0.7em][r]{$_{\blue{8}}1$} 
			& \makebox[0.7em][r]{$_{\blue{9}}2$} 
			& \makebox[0.7em][r]{$_{\blue{10}}1$} 
			& \makebox[0.7em][r]{$_{\blue{11}}2$} 
			& \makebox[0.7em][r]{$_{\blue{6}}\red{6}$} 
			& \makebox[0.7em][r]{$_{\blue{12}}4$} 
			& \makebox[0.7em][r]{$_{\blue{5}}\red{7}$} 
			& \makebox[0.7em][r]{$_{\blue{4}}3$} 
			& \makebox[0.7em][r]{$_{\blue{3}}2$}
			& \makebox[0.7em][r]{$_{\blue{2}}1$}
			& \makebox[1.4em][r]{$_{\blue{1}}6_{\blue{0}}$}
			& \makebox[1.1em][r]{$5_*$} \\
		\end{tabular}
	}
\end{table}
\end{center}
where the labels of the spaces are written as subscripts.
 Clearly, we have $y=1\in 3\Excp_3(w)$ and  all $3$-gap $3$-level excedance letters occurring  weakly to the right of $\alpha_y$ are given by $\alpha_{i_1}, \alpha_{i_2}$,  and $ \alpha_{i_3}$   with $i_1=1, i_2=6$, and $ i_3=8$ (see Figure \ref{case2-phi} for an illustration).  It  is straightforward  to verify  that $k=2$ is the smallest  integer satisfying  $\alpha_{i_k}<x_{i_{k+1}}+3$.  Then   $p=8$ is the smallest  integer such that $x_p=\alpha_{i_2}-2$.  Observe that all  non-$3$-gap-$3$-level excedance letters   located to the right of $\alpha_{p}=\alpha_8$ are given  by $\alpha_{j_1}, \alpha_{j_2}, \alpha_{j_3}, \alpha_{j_4}$ and $ \alpha_{j_5}$   with $j_1=9,  j_2=10$, $j_3=11, j_4=12$ and  $j_5=13$.   Then $u=\beta_1\beta_2\ldots \beta_{14}=\phi_{7, 3,3}(w, c)$ is a multipermutation obtained from $w$  by    replacing  $\alpha_{i_1}=\alpha_1$ with a   $7$, replacing   $\alpha_{i_{2}}=\alpha_{6}$  with $\alpha_{i_{1}}=\alpha_1$,  replacing $\alpha_{j_1}=\alpha_{9}$ with $\alpha_{i_2}=\alpha_6$,  and  replacing   $\alpha_{j_{z+1}}$ with $\alpha_{j_{z}}$ for all $1\leq z\leq 5 $  with the convention that $j_{6}=14$ 
 as illustrated in Figure \ref{case2-phi}.

\begin{figure}
 	\begin{center}

 	\begin{tikzpicture}[scale=1]

 	\draw[lightgray, thin] (1,0) -- (1, -2.2); 
 	\draw[lightgray, thin] (0,-0.5) -- (15,-0.5); 
 	
 	\node at (0.5, 0.25) {};
 	\node[red] at (1.5, 0.25) {$i_{1}$};
 	\node at (2.5, 0.25) {};
 	\node at (3.5, 0.25) {};
 	\node at (4.5, 0.25) {};
 	\node at (5.5, 0.25) {};
 	\node[red] at (6.5, 0.25) {$i_{2}$};
 	\node[red] at (7.5, 0.25) {};
 	\node[red] at (8.5, 0.25) {$i_{3}$};
 	\node[red] at (9.5, 0.25) {$j_{1}$};
 	\node[red] at (10.5, 0.25) {$j_{2}$};
 	\node[red] at (11.5, 0.25) {$j_{3}$};
 	\node[red] at (12.5, 0.25) {$j_{4}$};
 	\node[red] at (13.5, 0.25) {$j_{5}$};

 	\node at (0.5, -0.25) {$i$};
 	\node at (1.5, -0.25) {1};
 	\node at (2.5, -0.25) {2};
 	\node at (3.5, -0.25) {3};
 	\node at (4.5, -0.25) {4};
 	\node at (5.5, -0.25) {5};
 	\node at (6.5, -0.25) {6};
 	\node at (7.5, -0.25) {7};
 	\node at (8.5, -0.25) {8};
 	\node at (9.5, -0.25) {9};
 	\node at (10.5, -0.25) {10};
 	\node at (11.5, -0.25) {11};
 	\node at (12.5, -0.25) {12};
 	\node at (13.5, -0.25) {13};
 	\node at (14.5, -0.25) {14};
 	
 	\node at (0.5, -0.75) {$x_i$};
 	\node at (1.5, -0.75) {1};
 	\node at (2.5, -0.75) {1};
 	\node at (3.5, -0.75) {1};
 	\node at (4.5, -0.75) {2};
 	\node at (5.5, -0.75) {2};
 	\node at (6.5, -0.75) {2};
 	\node at (7.5, -0.75) {3};
 	\node at (8.5, -0.75) {4};
 	\node at (9.5, -0.75) {5};
 	\node at (10.5, -0.75) {5};
 	\node at (11.5, -0.75) {6};
 	\node at (12.5, -0.75) {6};
 	\node at (13.5, -0.75) {7};
 	\node at (14.5, -0.75) {7};
 	
 	\node at (0.5, -1.2) {$\alpha_i$};
 	\node[red] at (1.5, -1.2) {5};  
 	\node at (2.5, -1.2) {1};
 	\node at (3.5, -1.2) {2};
 	\node at (4.5, -1.2) {1};
 	\node at (5.5, -1.2) {2};
 	\node[red] at (6.5, -1.2) {6};  
 	\node at (7.5, -1.2) {4};  
 	\node at (8.5, -1.2) {7};
 	\node[red] at (9.5, -1.2) {3};  
 	\node[red] at (10.5, -1.2) {2}; 
 	\node[red] at (11.5, -1.2) {1}; 
 	\node[red] at (12.5, -1.2) {6}; 
 	\node[red] at (13.5, -1.2) {5};
 	\node at (14.5, -1.2) {};
 	
 	\begin{scope}[blue, line width=0.5pt, -{Latex[bend]}, shorten >=3pt, shorten <=3pt]
 		\draw[blue] (1.5, -1.45) to[bend right=10] (6.5, -1.45);
 		
 		\draw[blue] (6.5, -1.45) to[bend right=10] (9.5, -1.45);
 		
 		\draw[blue] (9.5, -1.45) to[bend right=10] (10.5, -1.45);
 		
 		\draw[blue] (10.5, -1.45) to[bend right=10] (11.5, -1.45);
 		
 		\draw[blue] (11.5, -1.45) to[bend right=10] (12.5, -1.45);
 		\draw[blue] (12.5, -1.45) to[bend right=10] (13.5, -1.45);
 		\draw[blue] (13.5, -1.45) to[bend right=10] (14.5, -1.45);
 	\end{scope}
 	
 	\node at (0.5, -1.9) {$\beta_i$};
 	\node at (1.5, -1.9) {7};
 	\node at (2.5, -1.9) {1};
 	\node at (3.5, -1.9) {2};
 	\node at (4.5, -1.9) {1};
 	\node at (5.5, -1.9) {2};
 	\node at (6.5, -1.9) {5};
 	\node at (7.5, -1.9) {4};
 	\node at (8.5, -1.9) {7};
 	\node at (9.5, -1.9) {6};
 	\node at (10.5, -1.9) {3};
 	\node at (11.5, -1.9) {2};
 	\node at (12.5, -1.9) {1};
 	\node at (13.5, -1.9) {6};
 	\node at (14.5, -1.9) {5};
 	
 \end{tikzpicture}
\end{center}
\caption{An example of Case 2 of the map $\phi_{n,g,h}$.}\label{case2-phi}
\end{figure}

Continuing with our running example with $c=11$, we have $y=5\notin 3\Excp_{3}(w)$.  One can easily check all   non-$3$-gap-$3$-level excedance letters of $w$  occurring   to the right of $\alpha_y$ are given by $\alpha_{j_1}, \alpha_{j_2}, \alpha_{j_3}, \alpha_{j_4}, \alpha_{j_5}$,  and $\alpha_{j_6}$  with $  j_1=7, j_2=9, j_3=10, j_4=11$, $j_5=12$ and $ j_6=13$ (see Figure  \ref{case3-phi} for an illustration). Upon applying the map $\phi_{7, 3,3}$ to $(w, c)$,   we  obtain  $u=\phi_{7, 3,3}(w, c)=\beta_1\beta_2\ldots\beta_{14}$ by  replacing $\alpha_y=\alpha_5$ with a $7$ and      replacing   $\alpha_{j_{z+1}}$ with $\alpha_{j_{z}}$ for all $0\leq z\leq 6 $  with the convention that $j_{7}=14$ and $ j_0=y$ as illustrated in Figure  \ref{case3-phi}.
\begin{figure}
	\begin{center}
 
		\begin{tikzpicture}[scale=1]
		
		\draw[lightgray, thin] (1,0) -- (1, -2.2); 
		\draw[lightgray, thin] (0,-0.5) -- (15,-0.5); 
		
		\node at (0.5, 0.25) {};
		\node at (1.5, 0.25) {};
		\node at (2.5, 0.25) {};
		\node at (3.5, 0.25) {};
		\node at (4.5, 0.25) {};
		\node at (5.5, 0.25) {};
		\node at (6.5, 0.25) {};
		\node[red] at (7.5, 0.25) {$j_{1}$};
		\node at (8.5, 0.25) {};
		\node[red] at (9.5, 0.25) {$j_{2}$};
		\node[red] at (10.5, 0.25) {$j_{3}$};
		\node[red] at (11.5, 0.25) {$j_{4}$};
		\node[red] at (12.5, 0.25) {$j_{5}$};
		\node[red] at (13.5, 0.25) {$j_{6}$};
		\node at (14.5, 0.25) {};
		
		\node at (0.5, -0.25) {$i$};
		\node at (1.5, -0.25) {1};
		\node at (2.5, -0.25) {2};
		\node at (3.5, -0.25) {3};
		\node at (4.5, -0.25) {4};
		\node at (5.5, -0.25) {5};
		\node at (6.5, -0.25) {6};
		\node at (7.5, -0.25) {7};
		\node at (8.5, -0.25) {8};
		\node at (9.5, -0.25) {9};
		\node at (10.5, -0.25) {10};
		\node at (11.5, -0.25) {11};
		\node at (12.5, -0.25) {12};
		\node at (13.5, -0.25) {13};
		\node at (14.5, -0.25) {14};
		
		\node at (0.5, -0.75) {$x_i$};
		\node at (1.5, -0.75) {1};
		\node at (2.5, -0.75) {1};
		\node at (3.5, -0.75) {1};
		\node at (4.5, -0.75) {2};
		\node at (5.5, -0.75) {2};
		\node at (6.5, -0.75) {2};
		\node at (7.5, -0.75) {3};
		\node at (8.5, -0.75) {4};
		\node at (9.5, -0.75) {5};
		\node at (10.5, -0.75) {5};
		\node at (11.5, -0.75) {6};
		\node at (12.5, -0.75) {6};
		\node at (13.5, -0.75) {7};
		\node at (14.5, -0.75) {7};
		
		\node at (0.5, -1.2) {$\alpha_i$};
		\node at (1.5, -1.2) {5};  
		\node at (2.5, -1.2) {1};
		\node at (3.5, -1.2) {2};
		\node at (4.5, -1.2) {1};
		\node[red] at (5.5, -1.2) {2};
		\node at (6.5, -1.2) {6};  
		\node[red] at (7.5, -1.2) {4};  
		\node at (8.5, -1.2) {7};
		\node[red] at (9.5, -1.2) {3};  
		\node[red] at (10.5, -1.2) {2}; 
		\node[red] at (11.5, -1.2) {1}; 
		\node[red] at (12.5, -1.2) {6}; 
		\node[red] at (13.5, -1.2) {5};
		\node at (14.5, -1.2) {};
		
		\begin{scope}[blue, line width=0.5pt, -{Latex[bend]}, shorten >=3pt, shorten <=3pt]
			
			\draw[blue] (5.5, -1.45) to[bend right=10] (7.5, -1.45);
			
			\draw[blue] (7.5, -1.45) to[bend right=10] (9.5, -1.45);
			
			\draw[blue] (9.5, -1.45) to[bend right=10] (10.5, -1.45);
			\draw[blue] (10.5, -1.45) to[bend right=10] (11.5, -1.45);
			
			\draw[blue] (11.5, -1.45) to[bend right=10] (12.5, -1.45);
			\draw[blue] (12.5, -1.45) to[bend right=10] (13.5, -1.45);
			\draw[blue] (13.5, -1.45) to[bend right=10] (14.5, -1.45);
		\end{scope}
		
		\node at (0.5, -1.9) {$\beta_i$};
		\node at (1.5, -1.9) {5};
		\node at (2.5, -1.9) {1};
		\node at (3.5, -1.9) {2};
		\node at (4.5, -1.9) {1};
		\node at (5.5, -1.9) {7};
		\node at (6.5, -1.9) {6};
		\node at (7.5, -1.9) {2};
		\node at (8.5, -1.9) {7};
		\node at (9.5, -1.9) {4};
		\node at (10.5, -1.9) {3};
		\node at (11.5, -1.9) {2};
		\node at (12.5, -1.9) {1};
		\node at (13.5, -1.9) {6};
		\node at (14.5, -1.9) {5};
		
	\end{tikzpicture}
\end{center}
\caption{An example of Case 3 of the map $\phi_{n,g,h}$.}\label{case3-phi}
\end{figure}

Now  we proceed to show that the map $\phi_{n,g,h}$ exhibits   the following celebrated properties.

\begin{lemma}\label{lem-phi-den1}
Let $(w, c)\in  \mathfrak{S}_{N'}\times  \{0,1, \ldots, m'\}$ with $w=\alpha_1\alpha_2\ldots \alpha_{m'+a-1}$ and  $\bar{w}=x_1x_2$ $\ldots x_{m'+a-1}$. Assume that $w$ has exactly $s$ $g$-gap $h$-level excedance places.
Choose the integer $y$ such that    the space before  the letter  $\alpha_y$ is labeled by $c$  under the   $g\den_h$-labeling  of $w$ when $c>0$,  and let $y=m'+1$ when $c=0$.   Then we have 
\begin{equation}\label{eq-Excp_r}
	g\Excp_h(\phi_{n,g,h}(w, c))=\left\{ \begin{array}{ll}
		g\Excp_h(w)&\, \mathrm{if}\,\,   0\leq c\leq s+\gamma_g  \\
		g\Excp_h(w)\cup \{y\}&\, \mathrm{otherwise}.
	\end{array}
	\right.
\end{equation}
Moreover, 
	if  $\alpha_j\neq n$ for all $j<y$, then we have 
	\begin{equation}\label{eq-den_r}
g\den_h(\phi_{n,g,h}(w, c))=g\den_h(w)+c.
	\end{equation}
	\end{lemma}
\begin{proof}
We shall proceed the proofs of  (\ref{eq-Excp_r})	and (\ref{eq-den_r}) by considering the following three cases.

\noindent{\bf Case 1:} $c=0$. 
In this case, the multipermutation $u=\phi_{n,g,h}(w, 0)$ is obtained from $w$ by inserting an $n$ immediately after $\alpha_{m'}$. It is straightforward to see that we have $ 
g\Excp_h(u)=g\Excp_h(w).
$ 	If $\alpha_j\neq n$ for all $j\leq m'$, then all the $n's$ occurs to the right of $\alpha_{m'}$. This implies that each letter lying to the right of $\alpha_{m'}$ is equal to $n$. Then inserting  an $n$ immediately after $\alpha_{m'}$ does not affect  $g\den_{h}(w)$, implying that 
$g\den_h(u)=g\den_h(w)$ as desired.

\noindent{\bf Case 2:}  $c>0$ and $y\notin g\Excp_h(w)$.\\
Recall that $u=\phi_{n,g,h}(w,c)$ is constructed  as follows.
	First, find all the  non-$g$-gap-$h$-level excedance letters of $w$ that occur    to  the right  of $\alpha_{y}$, say $\alpha_{j_1}, \alpha_{j_2}, \ldots, \alpha_{j_b}$ with $ j_1<j_2<\cdots<j_b$.  Then  replace $\alpha_{y}$ with an  $n$,   and replace  $\alpha_{j_{z+1}}$ with $\alpha_{j_{z}}$ for all $0\leq z\leq b $ with the convention that $j_{b+1}=m'+a$ and $j_0=y$.
	
	It is easily seen that  the procedure  from $w$ to $u$ keeps  all the $g$-gap $h$-level excedance letters of $w$  in   place.  Moreover,   this  procedure  also transforms $y$ into   a new  $g$-gap $h$-level excedance place when $y\leq m'-\gamma_g$ since for all $y\leq m'-\gamma_g$, we have $n\geq x_{y}+g$. 
	By the rules specified in the $g\den_h$-labeling of $w$, we have $c\leq \gamma_g$ when $y>m'-\gamma_g$,  and   $c\geq s+\gamma_g+1$ when $y\leq m'-\gamma_g$ and $y\notin g\Excp_h(w)$.  Therefore, we deduce that
	 $$
	 	g\Excp_h(u)=\left\{ \begin{array}{ll}
	 	g\Excp_h(w)&\, \mathrm{if}\,\,   1\leq c\leq s+\gamma_g  \\
	 	g\Excp_h(w)\cup \{y\}&\, \mathrm{otherwise},
	 \end{array}
	 \right.
	 $$
	as stated in (\ref{eq-Excp_r}). 
	       
Now we proceed to verify  (\ref{eq-den_r}).  We have two subcases. 

\noindent{\bf Subcase 2.1:} $y\leq  m'-\gamma_g$.\\
 Assume that there are exactly  $z$ $g$-gap $h$-level excedance letters that are located to the right of $\alpha_y$. 
By the rules   specified in the $g\den_h$-labeling of $w$,  one can easily check that $c=y+z+\gamma_g>s+\gamma_g$. In order to verify (\ref{eq-den_r}), it suffices to show that $g\den_h(u)=g\den_h(w)+y+z+\gamma_g$. 
Recall that 
$$
g\den_h(w)=\sum\limits_{i\in g\Excp_h(w)} (i+B^g_i(w)) +{\imv}(g\Exc_h(w))+{\inv}(g\Nexc_h (w)),
$$
	where $B^{g}_i(w)=|\{j \mid \alpha_i -g< x_j < \alpha_i\}|$. 
By (\ref{eq-Excp_r}), we have $g\Excp_h(u)=g\Excp_h(w)\cup \{y\}$.  By the construction of $u$, the procedure from $w$ to $u$ preserves all the non-$g$-gap-$h$-level excedance letters of $w$ as well as their relative order.  This implies that $g\Nexc_h(w)=g\Nexc_h(u)$.    Again by the construction of $u$, the procedure from $w$ to $u$ keeps all the $g$-gap $h$-level excedance letters of $w$ in place. 
Therefore, the procedure from $w$ to $u$  increases the sum 
$ 
\sum\limits_{i\in g\Excp_h(w)} (i+B^g_i(w))
$ by $y+B^g_y(u)=y+\gamma_g$ and does not affect ${\inv}(g\Nexc_h (w))$. Clearly, replacing $\alpha_y$ by an $n$ increases ${\imv}(g\Exc_h(w))$  by $z$ since there does not exist any $n's$ occurring to the left of $\alpha_y$ and there are exactly $z$ $g$-gap $h$-level excedance letters that are located to the right of $\alpha_y$ in $w$. Thus, we derive that $g\den_h(u)=g\den_h(w)+y+z+\gamma_g=g\den_h(w)+c$ as desired.

\noindent{\bf Subcase 2.2:} $y>  m'-\gamma_g$.\\
By the rules   specified in the $g\den_h$-labeling of $w$,  one can easily check that $c=m'+1-y\leq \gamma_g$. In order to verify (\ref{eq-den_r}), it suffices to show that $g\den_h(u)=g\den_h(w)+m'+1-y$. 
Recall that 
$$
g\den_h(w)=\sum\limits_{i\in g\Excp_h(w)} (i+B^g_i(w)) +{\imv}(g\Exc_h(w))+{\inv}(g\Nexc_h (w)).
$$
   By the construction of $u$, the procedure from $w$ to $u$ keeps all the $g$-gap $h$-level excedance letters in place.  By (\ref{eq-Excp_r}), we have $g\Excp_h(u)=g\Excp_h(w)$.  It follows that   the procedure from $w$ to $u$ does not affect the sum $\sum\limits_{i\in g\Excp_h(w)} (i+B^g_i(w))$ and $\imv(g\Exc_h(w))$.  
   Assume that there are exactly $z$ letters that are smaller than $n$  and   occur weakly to the  right of $\alpha_y$ in $w$. Then   replacing $\alpha_y$ by an $n$  would increase  ${\inv}(g\Nexc_h(w))$  by $z$. 
   Hence we have $g\den_h(u)=g\den_h(w)+z$. Recall that there are no $n's$ that are located  to the left of $\alpha_y$ in $u$. This implies that $y=m'+1-z$. Therefore, we derive that
   $$
   g\den_h(u)=g\den_h(w)+z=g\den_h(w)+m'+1-y=g\den_h(w)+c
   $$
   as desired.

\noindent{\bf Case 3:}  $c>0$ and $y\in g\Excp_h(w)$.\\
Recall that   $u=\beta_1\beta_2\ldots\beta_{m'+a}=\phi_{n,g,h}(w,c)$ is obtained  from $w$ as follows.
	\begin{itemize}
	\item  Find all the $g$-gap $h$-level excedance letters   of $w$  occurring      weakly to the right of   $\alpha_y$,  say
	$\alpha_{i_1}, \alpha_{i_2}, \ldots, \alpha_{i_a}$ with
	$ y=i_1<i_2<\cdots <i_a$. Find the smallest  integer $k$ satisfying $\alpha_{i_k}< x_{i_{k+1}}+g$   with the convention that $x_{i_{a+1}}=n$.
	\item    
	Choose the smallest integer  $p$ such that $x_p=\alpha_{i_k}-g+1$.	Find all the  non-$g$-gap-$h$-level excedance letters of $w$ that occur   weakly   to  the right  of $\alpha_{p}$, say $\alpha_{j_1}, \alpha_{j_2}, \ldots, \alpha_{j_b}$ with $ j_1<j_2<\cdots<j_b$.    
	\item   Replace  $\alpha_{i_1}=\alpha_y$ with an  $n$ and replace   $\alpha_{i_{z}}$  with $\alpha_{i_{z-1}}$  for all $1< z\leq k$. 
	
	\item  Replace  $\alpha_{j_1}$ with $\alpha_{i_k}$,  and  replace  $\alpha_{j_{z+1}}$ with $\alpha_{j_{z}}$ for all $1\leq z\leq b $  with the convention that $j_{b+1}=m'+a$.   
\end{itemize}

 Assume that there are exactly  $z$ $g$-gap  $h$-level excedance letters that are located to the right of $\alpha_y$. 
By the rules specified in the $g\den_h$-labeling of $w$, we have $c=z+1+\gamma_g\leq \gamma_g+s$.  In order to verify (\ref{eq-Excp_r}), it suffices to show that  $g\Excp_h(w)=g\Excp_h(u)$. 
According to the choices   of $k$ and $p$, we have $\alpha_{i_z}\geq  x_{i_{z+1}}+g$ for all $z<k$ and $\alpha_{i_k}< x_{j_1}+g$. This guarantees  that   the procedure from $w$ to $u$ keeps  all  the $g$-gap $h$-level excedance places, yielding that $g\Excp_h(w)=g\Excp_h(u)$ as desired.

Recall that $c=z+1+\gamma_g$. 
  In order to verify (\ref{eq-den_r}), it suffices to show that $g\den_h(u)=g\den_h(w)+z+1+\gamma_g$. 
Recall that 
$$
g\den_h(w)=\sum\limits_{i\in g\Excp_h(w)}(i+B^g_i(w)) +{\imv}(g\Exc_h(w))+{\inv}(g\Nexc_h (w)).
$$
Since $g\Excp_h(u)=g\Excp_h(w)$,  by the construction of $u$, the procedure from $w$ to $u$ transforms $\beta_{y}=n$ into a new $g$-gap $h$-level excedance letter, transforms $\alpha_{i_k}$ into a non-$g$-gap-$h$-level excedance letter, and preserves the other $g$-gap $h$-level excedance letters. Recall that 
$\beta_{i_j}=\alpha_{i_{j-1}}$ for all $1<j\leq k$. This implies that $B^g_{i_j}(u)=B^g_{i_{j-1}}(w)$ for all $1<j\leq k$. 
 Hence, we deduce that 
\begin{equation}\label{eq-sum}
	\begin{array}{lll}
\sum\limits_{i\in g\Excp_h(u)} (i+B^g_i(u))&=&B^g_{y}(u)-B^g_{i_k}(w)+\sum\limits_{i\in g\Excp_h(w)} (i+B^g_i(w))\\
&=&\gamma_g-B^g_{i_k}(w)+\sum\limits_{i\in g\Excp_h(w)} (i+B^g_i(w)),
\end{array}
\end{equation} 
  where  the second equality follows from the fact that  $B^g_y(u)=\gamma_g$.  
  Recall that there are no occurrences of $n$ to the left of $\alpha_y$. 
By the  equality $g\Excp_h(u)=g\Excp_h(w)$,  one can easily check that the procedure from $w$ to $u$ preserves the number of $g$-gap $h$-level excedance letters that are located to the right of $\alpha_y$. This implies that replacing $\alpha_{y}$ with  an $n$ would increase  ${\imv}(g\Exc_h(w))$ by $z$. 

Assume that there are exactly $t$ $g$-gap $h$-level excedance letters that are located to the left of $\alpha_{i_{k}}$ and  are greater than or equal to $\alpha_{i_k}$. Again by the choice  of $i_k$, all the $g$-gap $h$-level excedance letters located to the right of $\alpha_{i_k}$ are greater  than $\alpha_{i_k}$.  Recall that in the  the procedure from $w$ to $u$, we  transforms the letter  $\alpha_{i_k}$ into a non-$g$-gap-$h$-level excedance letter.    This procedure would decrease ${\imv}(g\Exc_h(w))$ by $t$.   Hence, we deduce that \begin{equation}\label{eq-imv}
	{\imv}(g\Exc_h(u))={\imv}(g\Exc_h(w))+z-t.
\end{equation}
In view of  (\ref{eq-sum}) and (\ref{eq-imv}), in order to prove $g\den_h(u)=g\den_h(w)+z+1+\gamma_g$, it remains to show that 
\begin{equation}\label{eq-inv}
	{\inv}(g\Nexc_h(u))={\inv}(g\Nexc_h(w))+t+1+B^g_{i_k}(w).
\end{equation} 
 Note that the procedure from $w$ to $u$  preserves    all the non-$g$-gap-$h$-level excedance letters of $w$ as well as their relative order,  and transforms  $\alpha_{i_k}$ into a non-$g$-gap-$h$-level excedance letter.  In order to prove (\ref{eq-inv}),  it suffices to  verify the following facts.

\noindent{\bf Fact 1:} All the non-$g$-gap-$h$-level excedance letters lying  to the left of $\beta_{j_{1}}$ are smaller than    $\beta_{j_{1}}=\alpha_{i_k}$ in $u$.
 
  \noindent{\bf Fact 2:} There are exactly $t+1+B^g_{i_k}(w)$ non-$g$-gap-$h$-level excedance letters that are located to the right of $\beta_{j_{1}}$ and   are smaller than  $\beta_{j_{1}}=\alpha_{i_k}$ in $u$.

  Since $i_k\in g\Excp_{h}(w)$,  it follows that all the non-$g$-gap-$h$-level excedance letters occurring  to the left of $\alpha_{i_{k}}$  are smaller than $\alpha_{i_k}$.  In order to prove Fact $1$, it suffices  to verify the following two claims. 
  
  \noindent{\bf Claim 1:} For all $i_k<j<p$, we have $j\notin g\Excp_h(w)$.\\
   If not,  choose $q$ be the smallest  integer such that $i_k<q<p$ and 
  $q\in g\Excp_h(w)$.  Then we have $i_{k+1}=q$. According to the choice of $p$, we have $x_q\leq \alpha_{i_k}-g$. Then we have $\alpha_{i_k}\geq x_q+g=x_{i_{k+1}}+g$,  yielding a contradiction with the choice of $i_k$ .  Hence, we conclude that $j\notin g\Excp_{h}(w)$ for all $i_k<j<p$ as claimed.
  
  \noindent{\bf Claim 2:} For all  $i_k<j<p$, we have $\alpha_j< \alpha_{i_k}$. \\
  If not, assume that $\alpha_q\geq \alpha_{i_k}$ for some $i_k<q<p$.  This combined with Claim  $1$ yields that  $\alpha_{i_k}\leq \alpha_q\leq x_{q}+g-1$, implying  that $x_q\geq \alpha_{i_k}+g-1$. This contradicts     the choice of $p$, completing  the proof of Claim $2$. 
  
   By the construction of $u$, in order to prove  Fact 2, it suffices to show that there are exactly $t+1+B^g_{i_k}(w)$ letters that are located weakly to the right of  $\alpha_{j_{1}}$ and  are smaller than $\alpha_{i_k}$ in $w$. 
  The following claim follows directly from   the selection of $j_{1}$. \\
  \noindent{\bf  Claim 3:} For all $p\leq j<j_{1}$, we have $j\in g\Excp_{h}(w)$. \\
   Assume that 
   $$
   u=|\{j\mid x_p\leq \alpha_{j}<\alpha_{i_k}, j<p\} |  
   $$
   and 
   $$
   v=|\{j\mid x_p\leq \alpha_{j}<\alpha_{i_k}, j\geq j_{1}\} | .
   $$
   Recall that $B^g_{i_k}(w)=|\{j\mid \alpha_{i_k}-g+1\leq x_{j}<\alpha_{i_k} \}|=|\{j\mid x_p\leq \alpha_{j}<\alpha_{i_k} \}|$.
   Claim  3 ensures that $u+v=B^g_{i_k}(w)$.  
   Recall that there are exactly $t$ $g$-gap $h$-level excedance letters that are located to the left of $\alpha_{i_k}$ and   are greater than or equal to $\alpha_{i_k}$ in $w$.  This   combined with Claim 2 tells us that  there are exactly $t+1$ letters that are greater than or equal to $\alpha_{i_k}$ and   are located to the left of $\alpha_{p}$ in $w$.    So there are exactly $u+t+1$ letters that are located to the left of $\alpha_p$ and  are greater than or equal to $x_p$. This implies that there are exactly $u+t+1$ letters that are located weakly to the right of  $\alpha_{p}$ and  are smaller than $x_p=\alpha_{i_k}-g+1$. Hence, there are exactly $u+t+1+v=t+1+B^g_{i_k}(w)$ letters that are located weakly to the right of  $\alpha_{p}$ and  are smaller than $\alpha_{i_k}$. 
    By Claim 3, such $t+1+B^g_{i_k}(w)$   letters  occur weakly to the right of $\alpha_{j_{1}}$ in $w$, completing the proof of Fact 2. 
	\end{proof}

  \begin{lemma}\label{lem-phi-den3}
 Let $(w, c)\in  \mathfrak{S}_{N'}\times  \{0,1, \ldots, m'\}$ with $w=\alpha_1\alpha_2\ldots \alpha_{m'+a-1}$ and  $\bar{w}=x_1x_2\ldots$ $ x_{m'+a-1}$.
  For all $g, \ell\geq 1$, if  $n\geq g+\ell\geq h\geq 1$, then we have
  	\begin{equation}\label{eq-gexc_l}
  	g\exc_\ell(\phi_{n,g,h}(w, c))=\left\{ \begin{array}{ll}
  			g\exc_\ell(w)&\, \mathrm{if}\,\,   0\leq c\leq g\exc_\ell(w) +\delta_\ell+\gamma_g  \\
  			g\exc_{\ell}(w)+1&\, \mathrm{otherwise}.\end{array}
  		\right.
  		\end{equation}

  	\end{lemma}
 \begin{proof}
 	Let $u=\beta_1\beta_2\ldots\beta_{m'+a}=\phi_{n,g,h}(w,c)$.  Assume that $w$ has exactly $s$ $g$-gap $h$-level excedance places.
 	Choose the integer $y$ such that    the space before  the letter  $\alpha_y$ is labeled by $c$  under the   $g\den_h$-labeling  of $w$ when $c>0$,  and let $y=m'+1$ when $c=0$.   Let $p$ be the smallest  integer such that $x_p=\ell$.   
 	Recall that $$g\exc_\ell(w)=|\{i\mid \alpha_i\geq x_i+g, x_i\geq \ell\}|,$$  
 	$$g\exc_\ell(u)=|\{i\mid \beta_i\geq x_i+g, x_i\geq \ell\}|, $$
 	$$
 	g\Excp_{h}(w)=\{i
 	\mid \alpha_i\geq x_i+g, \alpha_i\geq h\},
 	$$
 	and 
 	$$
 	g\Excp_{h}(u)=\{i
 	\mid \beta_i\geq x_i+g, \beta_i\geq h\}. 
 	$$
  Since $h\leq g+\ell$, one can easily check that    the number of $g$-gap $h$-level excedence letters occurring weakly  to the right of $\alpha_p$ (resp. $\beta_p$) in $u$ (resp. $w$) is given by $g\exc_\ell(w)$ (resp. $g\exc_\ell(u)$).

 Now we proceed to prove (\ref{eq-gexc_l}) by distinguishing the following cases. 
 
 \noindent{\bf Case 1:} $c\leq s+\gamma_g$.\\
 	From  (\ref{eq-Excp_r}), it follows that  $g\Excp_h(w)= g\Excp_h(u)$, implying that  $g\exc_\ell(w)= g\exc_\ell(u)$.
 	
 	 \noindent{\bf Case 2:} $s+\gamma_g< c \leq  g\exc_\ell(w)+\delta_\ell+\gamma_g$.\\ 
 	 By the rules specified in the $g\den_h$-labeling of $w$, we must have  $y<p$. By (\ref{eq-Excp_r}), we have $g\Excp_h(u)= g\Excp_h(w)\cup\{y\}$. Since $y<p$,   it follows that $g\exc_\ell(w)= g\exc_\ell(u)$ as desired.

 	 \noindent{\bf Case 3:} $  c >  g\exc_\ell(w)+\delta_\ell+\gamma_g$.\\ 
 		 By the rules specified in the $g\den_h$-labeling of $w$, we must have  $y\geq p$. Again by (\ref{eq-Excp_r}), we have $g\Excp_h(u)= g\Excp_h(w)\cup\{y\}$. Since $y\geq p$,   it follows that $g\exc_\ell(u)= g\exc_\ell(w)+1$ as desired, completing the proof. 
 	\end{proof}
  \begin{framed}
 	\begin{center}
 		{\bf The map $\psi_{n,g,h}:\mathfrak{S}_{N} \longrightarrow \mathfrak{S}_{N'}\times \{0,1,\ldots, m'\} $}
 	\end{center}
 	Let    $u=\beta_1\beta_2\ldots \beta_{m'+a}$ and $\bar{u}=x_1x_2\ldots x_{m'+a}$.  
 	 Assume that $\beta_y$ is the leftmost $n$ of $u$.  If $y=m'+1$, define $\psi_{n,g,h}(u)=(w,0)$ where $w$ is obtained from $u$ by removing $\beta_{y}$.  Otherwise,  
 	 find all the non-$g$-gap-$h$-level excedance letters  occurring to the right of $\beta_y$, say $\beta_{j_1}, \beta_{j_2}, \ldots, \beta_{j_a}$ with $j_1<j_2<\cdots<j_a$.  Find the greatest integer $\ell$  satisfying  $\beta_{j_\ell}\geq x_{j_{\ell-1}}+g$ and $\beta_{j_{\ell}}\geq h$ with the convention that $j_0=y$. If such $\ell$ does not exist, set $\ell=0$.  
 	 Generate a multipermutation $w$   from $u$ by distinguishing the following two cases. 
 	 \begin{itemize}
 		 
 		 \item Case 1: $\ell=0$\\
 		 Replace $\beta_{j_z}$ with $\beta_{j_{z+1}}$ for all $0\leq z<a$ with the convention that $j_0=y$. 
 		 
 		 \item Case 2: $\ell\geq 1$. 
 		  \begin{itemize}
 		 	\item    Choose the smallest  integer $p$ satisfying $x_p=\beta_{j_\ell}-g+1$.
 		 	\item Find all the $g$-gap $h$-level excedance letters  of $u$  located     between   $\beta_y$ and $\beta_{p}$ (including $\beta_y$),  say  $\beta_{i_1}, \beta_{i_2}, \ldots, \beta_{i_k}$ with    
 		 	 $ y=i_1<i_2<\ldots <i_k$.    
 		  
 		 	\item  Replace   $\beta_{i_{z}}$  with $\beta_{i_{z+1}}$  for all $1\leq z\leq k$ with the convention that $i_{k+1}=j_\ell$.
 		 	
 		 	\item   Replace  $\beta_{j_{z}}$ with $\beta_{j_{z+1}}$ for all $\ell\leq z<a $. 
 		 \end{itemize}

 	\end{itemize}
 	Define $\psi_{n,g,h}(u)=(w,c)$, where  $c=g\den_{h}(u)-g\den_{h}(w)$. 
 	
 \end{framed}

Take  $N=\{1^3, 2^3, 3, 4, 5^2, 6^2, 7^{2}\}$,    $u=51217627432165$ and $g=h=3$ for example.  
Clearly, $\beta_5$ is the leftmost $7$ of $u$, and hence we have $y=5$. 
It is not difficult to see that all  non-$3$-gap-$3$-level excedance letters occurring  to the right of $\beta_y$ are given by $\beta_{j_1}, \beta_{j_2},  \ldots, \beta_{j_7} $ with $j_1=7, j_2=9, j_3=10$,  $  j_4=11$, $j_5=12$, $j_6=13$ and $j_7=14$.  Clearly, we have $\beta_{j_z}\leq x_{j_{z-1}}+2$ for all $1\leq z\leq 7$ with the convention $j_0=y=5$.  This implies that $\ell=0$. Then we obtain a multipermutation $w=\alpha_1\alpha_2\ldots \alpha_{13}$  from $u$ by     replacing $\beta_{j_z}$ with $\beta_{j_{z+1}}$ for all $0\leq z\leq 6$ as demonstrated in Figure \ref{case1-psi}. Then we have $c=3\den_3(u)-3\den_3(w)=11$. 
\begin{figure}
	\begin{center}
 
		\begin{tikzpicture}[scale=1]
		
		\draw[lightgray, thin] (1,0) -- (1, -2.2); 
		\draw[lightgray, thin] (0,-0.5) -- (15,-0.5); 
		
		\node at (0.5, 0.25) {};
		\node at (1.5, 0.25) {};
		\node at (2.5, 0.25) {};
		\node at (3.5, 0.25) {};
		\node at (4.5, 0.25) {};
		\node at (5.5, 0.25) {};
		\node at (6.5, 0.25) {};
		\node[red] at (7.5, 0.25) {$j_{1}$};
		\node at (8.5, 0.25) {};
		\node[red] at (9.5, 0.25) {$j_{2}$};
		\node[red] at (10.5, 0.25) {$j_{3}$};
		\node[red] at (11.5, 0.25) {$j_{4}$};
		\node[red] at (12.5, 0.25) {$j_{5}$};
		\node[red] at (13.5, 0.25) {$j_{6}$};
		\node[red] at (14.5, 0.25) {$j_{7}$};
		
		\node at (0.5, -0.25) {$i$};
		\node at (1.5, -0.25) {1};
		\node at (2.5, -0.25) {2};
		\node at (3.5, -0.25) {3};
		\node at (4.5, -0.25) {4};
		\node at (5.5, -0.25) {5};
		\node at (6.5, -0.25) {6};
		\node at (7.5, -0.25) {7};
		\node at (8.5, -0.25) {8};
		\node at (9.5, -0.25) {9};
		\node at (10.5, -0.25) {10};
		\node at (11.5, -0.25) {11};
		\node at (12.5, -0.25) {12};
		\node at (13.5, -0.25) {13};
		\node at (14.5, -0.25) {14};
		
		\node at (0.5, -0.75) {$x_i$};
		\node at (1.5, -0.75) {1};
		\node at (2.5, -0.75) {1};
		\node at (3.5, -0.75) {1};
		\node at (4.5, -0.75) {2};
		\node at (5.5, -0.75) {2};
		\node at (6.5, -0.75) {2};
		\node at (7.5, -0.75) {3};
		\node at (8.5, -0.75) {4};
		\node at (9.5, -0.75) {5};
		\node at (10.5, -0.75) {5};
		\node at (11.5, -0.75) {6};
		\node at (12.5, -0.75) {6};
		\node at (13.5, -0.75) {7};
		\node at (14.5, -0.75) {7};
		
		\node at (0.5, -1.2) {$\beta_i$};
		\node at (1.5, -1.2) {5};  
		\node at (2.5, -1.2) {1};
		\node at (3.5, -1.2) {2};
		\node at (4.5, -1.2) {1};
		\node[red] at (5.5, -1.2) {7};
		\node at (6.5, -1.2) {6};  
		\node[red] at (7.5, -1.2) {2};  
		\node at (8.5, -1.2) {7};
		\node[red] at (9.5, -1.2) {4};  
		\node[red] at (10.5, -1.2) {3}; 
		\node[red] at (11.5, -1.2) {2}; 
		\node[red] at (12.5, -1.2) {1}; 
		\node[red] at (13.5, -1.2) {6};
		\node[red] at (14.5, -1.2) {5};
		
		\begin{scope}[blue, line width=0.5pt, -{Latex[bend]}, shorten >=3pt, shorten <=3pt]
			\draw[blue] (14.5, -1.45) to[bend left=10] (13.5, -1.45);
			\draw[blue] (13.5, -1.45) to[bend left=10] (12.5, -1.45);
			
			\draw[blue] (12.5, -1.45) to[bend left=10] (11.5, -1.45);
			
			\draw[blue] (11.5, -1.45) to[bend left=10] (10.5, -1.45);
			
			\draw[blue] (10.5, -1.45) to[bend left=10] (9.5, -1.45);
			
			\draw[blue] (9.5, -1.45) to[bend left=10] (7.5, -1.45);
			\draw[blue] (7.5, -1.45) to[bend left=10] (5.5, -1.45);

		\end{scope}
		
		\node at (0.5, -1.9) {$\alpha_i$};
		\node at (1.5, -1.9) {5};
		\node at (2.5, -1.9) {1};
		\node at (3.5, -1.9) {2};
		\node at (4.5, -1.9) {1};
		\node at (5.5, -1.9) {2};
		\node at (6.5, -1.9) {6};
		\node at (7.5, -1.9) {4};
		\node at (8.5, -1.9) {7};
		\node at (9.5, -1.9) {3};
		\node at (10.5, -1.9) {2};
		\node at (11.5, -1.9) {1};
		\node at (12.5, -1.9) {6};
		\node at (13.5, -1.9) {5};
		\node at (14.5, -1.9) {};
		
	\end{tikzpicture}
\end{center}
\caption{An example of Case 1 of the map $\psi_{n,g,h}$.} \label{case1-psi}
\end{figure}

 Take  $N=\{1^3, 2^3, 3, 4, 5^2, 6^2, 7^{2}\}$,    $u=71212547632165$ and $g=h=3$ for example. 
 Clearly, $\beta_1$ is the leftmost $7$ of $u$, and hence we have $y=1$. 
 It is not difficult to see that all  non-$3$-gap-$3$-level excedance letters occurring  to the right of $\beta_y$ are given by $\beta_{j_1}, \beta_{j_2},  \ldots, \beta_{j_{11}} $ with $j_1=2, j_2=3, j_3=4$,  $  j_4=5$, $j_5=7$, $j_6=9$,  $j_7=10$, $j_8=11$, $j_9=12$, $j_{10}=13$ and $j_{11}=14$ (see Figure \ref{case2-psi} for an illustration).  Clearly,  $\ell=6$ is the greatest integer satisfying  $\beta_{j_\ell}\geq x_{j_{\ell-1}}+3$ and $\beta_{j_\ell}\geq 3$. 
 Then   	 $p=8$ is the smallest  integer satisfying $x_p=\beta_{j_{\ell}}-2=4$. 
 Moreover, one can easily check that all the $3$-gap $3$-level excedance letters lying between $\beta_{y}$ and $\beta_{p}$ (including $\beta_y$) are given by $\beta_{i_1}$ and  $ \beta_{i_2}$  with  $i_1=1$ and $i_2=6$.    
Then generate a multipermutation $w=\alpha_1\alpha_2\ldots \alpha_{13}$  from $u$  by replacing $\beta_y=\beta_1$ with $\beta_{i_2}=\beta_{6}$, replacing $\beta_{i_2}=\beta_{6}$ with $\beta_{j_6}=\beta_{9}$,   and replacing $\beta_{j_z}$ with $\beta_{j_{z+1}}$ for all $6\leq z\leq 10$ as illustrated in Figure \ref{case2-psi}. 
Then we have $c=3\den_3(u)-3\den_3(w)=7$. 
\begin{figure}
		\begin{center}
	 
		\begin{tikzpicture}[scale=1]
		
		\draw[lightgray, thin] (1,0) -- (1, -2.2); 
		\draw[lightgray, thin] (0,-0.5) -- (15,-0.5); 
		
		\node at (0.5, 0.25) {};
		\node[red] at (1.5, 0.25) {$i_{1}$};
		\node[red] at (2.5, 0.25) {$j_{1}$};
		\node[red] at (3.5, 0.25) {$j_{2}$};
		\node[red] at (4.5, 0.25) {$j_{3}$};
		\node[red] at (5.5, 0.25) {$j_{4}$};
		\node[red] at (6.5, 0.25) {$i_{2}$};
		\node[red] at (7.5, 0.25) {$j_{5}$};
		\node[red] at (8.5, 0.25) {$p$};
		\node[red] at (9.5, 0.25) {$j_{6}$};
		\node[red] at (10.5, 0.25) {$j_{7}$};
		\node[red] at (11.5, 0.25) {$j_{8}$};
		\node[red] at (12.5, 0.25) {$j_{9}$};
		\node[red] at (13.5, 0.25) {$j_{10}$};
		\node[red] at (14.5, 0.25) {$j_{11}$};

		\node at (0.5, -0.25) {$i$};
		\node at (1.5, -0.25) {1};
		\node at (2.5, -0.25) {2};
		\node at (3.5, -0.25) {3};
		\node at (4.5, -0.25) {4};
		\node at (5.5, -0.25) {5};
		\node at (6.5, -0.25) {6};
		\node at (7.5, -0.25) {7};
		\node at (8.5, -0.25) {8};
		\node at (9.5, -0.25) {9};
		\node at (10.5, -0.25) {10};
		\node at (11.5, -0.25) {11};
		\node at (12.5, -0.25) {12};
		\node at (13.5, -0.25) {13};
		\node at (14.5, -0.25) {14};
		
		\node at (0.5, -0.75) {$x_i$};
		\node at (1.5, -0.75) {1};
		\node at (2.5, -0.75) {1};
		\node at (3.5, -0.75) {1};
		\node at (4.5, -0.75) {2};
		\node at (5.5, -0.75) {2};
		\node at (6.5, -0.75) {2};
		\node at (7.5, -0.75) {3};
		\node at (8.5, -0.75) {4};
		\node at (9.5, -0.75) {5};
		\node at (10.5, -0.75) {5};
		\node at (11.5, -0.75) {6};
		\node at (12.5, -0.75) {6};
		\node at (13.5, -0.75) {7};
		\node at (14.5, -0.75) {7};
		
		\node at (0.5, -1.2) {$\beta_i$};
		\node[red] at (1.5, -1.2) {7};  
		\node at (2.5, -1.2) {1};
		\node at (3.5, -1.2) {2};
		\node at (4.5, -1.2) {1};
		\node at (5.5, -1.2) {2};
		\node[red] at (6.5, -1.2) {5};  
		\node at (7.5, -1.2) {4};  
		\node at (8.5, -1.2) {7};
		\node[red] at (9.5, -1.2) {6};  
		\node[red] at (10.5, -1.2) {3}; 
		\node[red] at (11.5, -1.2) {2}; 
		\node[red] at (12.5, -1.2) {1}; 
		\node[red] at (13.5, -1.2) {6};
		\node[red] at (14.5, -1.2) {5};
		
		\begin{scope}[blue, line width=0.5pt, -{Latex[bend]}, shorten >=3pt, shorten <=3pt]
			\draw[blue] (14.5, -1.45) to[bend left=10] (13.5, -1.45);
			\draw[blue] (13.5, -1.45) to[bend left=10] (12.5, -1.45);
			
			\draw[blue] (12.5, -1.45) to[bend left=10] (11.5, -1.45);
			
			\draw[blue] (11.5, -1.45) to[bend left=10] (10.5, -1.45);
			
			\draw[blue] (10.5, -1.45) to[bend left=10] (9.5, -1.45);
			
			\draw[blue] (9.5, -1.45) to[bend left=10] (6.5, -1.45);
			\draw[blue] (6.5, -1.45) to[bend left=10] (1.5, -1.45);

		\end{scope}
		
		\node at (0.5, -1.9) {$\alpha_i$};
		\node at (1.5, -1.9) {5};
		\node at (2.5, -1.9) {1};
		\node at (3.5, -1.9) {2};
		\node at (4.5, -1.9) {1};
		\node at (5.5, -1.9) {2};
		\node at (6.5, -1.9) {6};
		\node at (7.5, -1.9) {4};
		\node at (8.5, -1.9) {7};
		\node at (9.5, -1.9) {3};
		\node at (10.5, -1.9) {2};
		\node at (11.5, -1.9) {1};
		\node at (12.5, -1.9) {6};
		\node at (13.5, -1.9) {5};
		\node at (14.5, -1.9) {};
		
	\end{tikzpicture}
\end{center}
	\caption{An example of Case 2 of the map $\psi_{n,g,h}$.} \label{case2-psi}
\end{figure}

Now we proceed to  prove that the map   $\psi_{n,g,\ell}$ verifies the following properties. 
   
 \begin{lemma}\label{lem-psi-den1}
 	Let $(w,c)\in \mathfrak{S}_{N'}\times\{ 0,1,\ldots, m'\}$ with $w=\alpha_1\alpha_2\ldots \alpha_{m'+a-1}$.  Choose the integer $y$ such that    the space before  the letter  $\alpha_y$ is labeled by $c$  under the   $g\den_h$-labeling  of $w$ when $c>0$,  and let $y=m'+1$ when $c=0$.   If   $\alpha_j\neq n$ for all $j<y$, then we have $\psi_{n,g,h}(\phi_{n,g,h}(w, c))=(w,c)$.  
 	\end{lemma}
 \begin{proof}  
 	We have three cases.\\
 		\noindent{\bf Case 1:} $c=0$.\\
 	By  the definitions of $\phi_{n,g,h}$ and $\psi_{n,g,h}$, one can easily check that $\psi_{n,g,h}(\phi_{n,g,h}(w, 0))=(w,0). $
 	\noindent{\bf Case 2:}  $c>0$ and $y\notin g\Excp_h(w)$.\\
 	Recall that $u=\beta_1\beta_2\cdots \beta_{m'+a}=\phi_{n,g,h}(w,c)$ is constructed  as follows.
 	First, find all the  non-$g$-gap-$h$-level excedance letters of $w$ lying    weakly to  the right  of $\alpha_{y}$, say $\alpha_{j_1}, \alpha_{j_2}, \ldots, \alpha_{j_b}$ with $ y=j_1<j_2<\cdots<j_b$.  Then  replace $\alpha_{j_1}$ with an  $n$,   and replace  $\alpha_{j_{z+1}}$ with $\alpha_{j_{z}}$ for all $1\leq z\leq b $ with the convention that $j_{b+1}=m'+a$. 
 	It is easily seen that $\beta_{y}=n$ is the leftmost $n$ in $u$ and   the non-$g$-gap-$h$-level excedance letters  lying to the right of $\beta_{y}$ are  given by $\beta_{j_2}, \beta_{j_3}, \ldots, \beta_{j_{b+1}}$. Clearly, for all $2\leq z\leq b+1$, we have  either $\beta_{j_z}=\alpha_{j_{z-1}}< x_{j_{z-1}}+g$ or $\beta_{j_z}=\alpha_{j_{z-1}}<h$.    
 	Then by the definition of $\psi_{n,g,h}$,  we have $\psi_{n,g,h}(u)=(w', c')$ where $w'$ is obtained from  $u$ by replacing $\beta_{j_z}$ with $\beta_{j_{z+1}}$ for all $1\leq z\leq b$ and $c'=g\den_{h}(u)-g\den_{h}(w')$. 
 	By (\ref{eq-den_r}), we have $g\den_h(u)=g\den_h(w)+c$. In order to prove that $\psi_{n,g,h}(\phi_{n,g,h}(w,c))=(w,c)$, it suffices to show that $w=w'$. This can be justified by the fact that 
  $\beta_{j_{z+1}}=\alpha_{j_z}$ for all $1\leq z\leq b$.

 		\noindent{\bf Case 3:}   $c>0$ and $y\in g\Excp_h(w)$.\\
 		Assume that  $w$ has exactly  $s$ $g$-gap $h$-level excedance letters. By the rules specified in the $g\den_h$-labeling of $w$,   we have $c\leq s+\gamma_g$.   By (\ref{eq-Excp_r}), we have $g\Excp_h(w)=g\Excp_h(u)$. 
 		
 	 Recall that   $u$ is obtained  from $w$ by carrying out  the following procedure.
 	\begin{itemize}
 	\item  Find all the $g$-gap $h$-level excedance letters   of $w$  occurring      weakly to the right of   $\alpha_y$,  say
 	$\alpha_{i_1}, \alpha_{i_2}, \ldots, \alpha_{i_a}$ with
 	$ y=i_1<i_2<\cdots <i_a$. Find the smallest  integer $k$ satisfying $\alpha_{i_k}< x_{i_{k+1}}+g$   with the convention that $x_{i_{a+1}}=n$.
 	\item    
 	Choose the smallest integer  $p$ such that $x_p=\alpha_{i_k}-g+1$.	Find all the  non-$g$-gap-$h$-level excedance letters of $w$ that occur   weakly   to  the right  of $\alpha_{p}$, say $\alpha_{j_1}, \alpha_{j_2}, \ldots, \alpha_{j_b}$ with $ j_1<j_2<\cdots<j_b$.    
 	\item   Replace  $\alpha_{i_1}=\alpha_y$ with an  $n$ and replace   $\alpha_{i_{z}}$  with $\alpha_{i_{z-1}}$  for all $1< z\leq k$. 
 	
 	\item  Replace  $\alpha_{j_1}$ with $\alpha_{i_k}$,  and  replace  $\alpha_{j_{z+1}}$ with $\alpha_{j_{z}}$ for all $1\leq z\leq b $  with the convention that $j_{b+1}=m'+a$.   
 \end{itemize}
 	It is easy to check that $\beta_{y}=n$ is the leftmost $n$ in $u$.  Since $g\Excp_h(w)=g\Excp_h(u)$,  the non-$g$-gap-$h$-level excedance letters that are located weakly to the  right of $\beta_p$ are given by $\beta_{j_1}, \beta_{j_2}, \ldots, \beta_{j_{b+1}}$.    Recall that $p$ is the smallest integer such that $x_p=\alpha_{i_k}-g+1=\beta_{j_1}-g+1$.  Moreover, we have  either $\beta_{j_{z+1}}=\alpha_{j_z}< x_{j_{z}}+g$ or $\beta_{j_{z+1}}=\alpha_{j_z}<h$ for all $1\leq z\leq b$.
 	Moreover, for any $q<p$, we have $\beta_{j_1}=\alpha_{i_k}\geq x_q+g$ and $\beta_{j_1}\geq h$.   Since $g\Excp_h(w)=g\Excp_h(u)$, it follows that   all the $g$-gap $h$-level excedance letters located between $\beta_{y}$ and $\beta_p$ (including $\beta_y$) are given by $\beta_{i_1}, \beta_{i_2}, \ldots, \beta_{i_k}$.  Let $\psi_{n,g,h}(u)=(w', c')$, where $c'=g\den_h(u)-g\den_h(w')$ and $w'$ is obtained form $u$ by carrying out the following procedures. 
 	 	\begin{itemize}
 	 	\item  Replace   $\beta_{i_{z}}$  with $\beta_{i_{z+1}}$  for all $1\leq z< k$.
 	 	\item Replace $\beta_{i_k}$ with $\beta_{j_1}$.
 	 	\item   Replace  $\beta_{j_{z}}$ with $\beta_{j_{z+1}}$ for all $1\leq z\leq b $. 
 	 \end{itemize}
 	Again by the definition of $\phi_{n,g,h}$, we have $\beta_{i_{z+1}}=\alpha_{i_z}$ for all $1\leq z\leq k-1$,  $\beta_{j_{1}}=\alpha_{i_k}$, and $\beta_{j_{z+1}}=\alpha_{j_z}$ for all $1\leq z\leq b$. This ensures that $w'=w$. 
 		By (\ref{eq-den_r}), we have $g\den_h(u)=g\den_h(w)+c$.  Therefore, we deduce that  $\psi_{n,g,h}(\phi_{n,g,h}(w, c))=(w', c')=(w, c)$ as desired, completing  the proof. 
  \end{proof}

 \subsection{Finishing the proof of Theorem \ref{th-Phi-den}}
 Now we are in  the position to establish the map $\Phi^{\den}_{g,h}$.

  \begin{framed}
 	\begin{center}
 		{\bf The map $\Phi^{\den}_{g,h}:\mathfrak{S}_{M'}\times \mathcal{P}(k_n, m') \longrightarrow \mathfrak{S}_{M}$ }
 	\end{center}
 	Given  $(w, \lambda)\in \mathfrak{S}_{M'}\times \mathcal{P}(k_n, m')  $, let $\lambda^{(0)}=\lambda=(\lambda_1, \lambda_2, \ldots, \lambda_{k_n})$ and  $w^{(0)}=w$. Set $b=0$ and carry out the following procedure. 
 	 	
 		(A) Define 
 		$$T^{(b)}=(a^{(b)}_{1}, a^{(b)}_2, \ldots,  a^{(b)}_{m'+1})$$ where $a^{(b)}_{i}$  is label of the place immediately before the $i$-th letter of $w^{(b)}$ (counting from left to right) under the $g\den_h$-labeling of $w^{(i)}$ for all $1\leq i\leq m'$ and  $a^{(b)}_{m'+1}$ is the label of the place immediately after   the $m'$-th letter    of $w^{(b)}$ under the $g\den_h$-labeling of $w^{(b)}$. 
Let $y_b$ be the greatest integer such that $a^{(b)}_{y_b}\in \lambda^{(b)}$. 
 	Let $w^{(b+1)}=\phi_{n,g,h}(w^{(b)}, a^{(b)}_{y_b})$ and let $\lambda^{(b+1)}=\lambda^{(b)}\setminus \{a^{(b)}_{y_b}\}$.  
 	
 	(B)  Replace $b$ by $b+1$. If $b=k_n$, then we stop. Otherwise, we go back to (A). 
 	
 	Define $\Phi^{\den}_{g,h}(w, \lambda)=w^{(k_n)}$.

 \end{framed}

 For example, let $M=\{1^3, 2^3, 3^2, 4^2, 5^4\}$ and $w=3224143121$. Take $g=2$, $h=3$, and $\lambda=(9,9,5,4)$, we will generate $w^{(4)}=32245551543121=\Phi^{\den}_{2,3}(w, \lambda)$ as follows.
  	
	\begin{table}[h]
	\centering
	\begin{minipage}{0.45\textwidth}
		\centering
		\vspace{0.3em}
		\fontsize{11}{9}\selectfont  
		\renewcommand{\arraystretch}{1.5}  
		\setlength{\tabcolsep}{0.52em} 
		\begin{tabular}{c|cccccccccccc}
			$i$ & 1 & 2 & 3 & 4 & 5 & 6 & 7 & 8 & 9 & 10  \\
			\hline
			\multirow{2}{*}{$\begin{array}{c} x_i \\ \alpha_i \end{array}$} 
			& \makebox[0.7em][r]{1} & \makebox[0.7em][r]{1} & \makebox[0.7em][r]{1} & \makebox[0.7em][r]{2} & \makebox[0.7em][r]{2} & \makebox[0.7em][r]{2} & \makebox[0.7em][r]{3} & \makebox[0.7em][r]{3} & \makebox[0.7em][r]{4} & \makebox[0.7em][r]{4}  \\
			& \makebox[0.7em][r]{$_{\green{5}}\red{3}$} 
			& \makebox[0.7em][r]{$_{\green{6}}2$} 
			& \makebox[0.7em][r]{$_{\green{7}}2$} 
			& \makebox[0.7em][r]{$_{\green{4}}\red{4}$} 
			& \makebox[0.7em][r]{$_{\green{8}}1$} 
			& \makebox[0.7em][r]{$_{\green{3}}\red{4}$} 
			& \makebox[0.7em][r]{$_{\green{9}}3$} 
			& \makebox[0.7em][r]{$_{\green{10}}1$} 
			& \makebox[0.7em][r]{$_{\green{2}}2$} 
			& \makebox[0.7em][r]{$_{\green{1}}1$}
			& \makebox[0em][r]{$_{\green{0}}$} \\
		\end{tabular}
		\\[0.5em] 
		{\small $w^{(0)}$} 
	\end{minipage}
	\hfill
	\begin{minipage}{0.45\textwidth}
		\centering
		\vspace{0.3em}
		\fontsize{11}{9}\selectfont  
		\renewcommand{\arraystretch}{1.5}  
		\setlength{\tabcolsep}{0.4em} 
		\begin{tabular}{c|ccccccccccccc}
			$i$ & 1 & 2 & 3 & 4 & 5 & 6 & 7 & 8 & 9 & 10 & 11 \\
			\hline
			\multirow{2}{*}{$\begin{array}{c} x_i \\ \alpha_i \end{array}$} 
			& \makebox[0.7em][r]{1} & \makebox[0.7em][r]{1} & \makebox[0.7em][r]{1} & \makebox[0.7em][r]{2} & \makebox[0.7em][r]{2} & \makebox[0.7em][r]{2} & \makebox[0.7em][r]{3} & \makebox[0.7em][r]{3} & \makebox[0.7em][r]{4} & \makebox[0.7em][r]{4} & \makebox[0.7em][r]{5} \\
			& \makebox[0.7em][r]{$_{\green{6}}\red{3}$}
			& \makebox[0.7em][r]{$_{\green{7}}2$} 
			& \makebox[0.7em][r]{$_{\green{8}}2$} 
			& \makebox[0.7em][r]{$_{\green{5}}\red{4}$} 
			& \makebox[0.7em][r]{$_{\green{9}}1$} 
			& \makebox[0.7em][r]{$_{\green{4}}\red{4}$} 
			& \makebox[0.7em][r]{$_{\green{3}}\red{5}$} 
			& \makebox[0.7em][r]{$_{\green{10}}3$} 
			& \makebox[0.7em][r]{$_{\green{2}}1$} 
			& \makebox[0.7em][r]{$_{\green{1}}2$} 
			& \makebox[0.7em][r]{$_{\green{0}}1$}
			& \makebox[0em][r]{$_*$} \\
		\end{tabular}
		\\[0.5em] 
		{\small $w^{(1)}$} 
	\end{minipage}
\end{table}

\begin{table}[h]
	\centering
	\begin{minipage}{0.45\textwidth}
		\centering
		\vspace{0.3em}
		\fontsize{11}{9}\selectfont  
		\renewcommand{\arraystretch}{1.5}  
		\setlength{\tabcolsep}{0.38em} 
		\begin{tabular}{c|cccccccccccccc}
			$i$ & 1 & 2 & 3 & 4 & 5 & 6 & 7 & 8 & 9 & 10 &11&12 \\
			\hline
			\multirow{2}{*}{$\begin{array}{c} x_i \\ \alpha_i \end{array}$} 
			& \makebox[0.7em][r]{1} & \makebox[0.7em][r]{1} & \makebox[0.7em][r]{1} & \makebox[0.7em][r]{2} & \makebox[0.7em][r]{2} & \makebox[0.7em][r]{2} & \makebox[0.7em][r]{3} & \makebox[0.7em][r]{3} & \makebox[0.7em][r]{4} & \makebox[0.7em][r]{4} & \makebox[0.7em][r]{5} & \makebox[0.7em][r]{5} \\
			& \makebox[0.7em][r]{$_{\green{6}}\red{3}$}
			& \makebox[0.7em][r]{$_{\green{7}}2$} 
			& \makebox[0.7em][r]{$_{\green{8}}2$} 
			& \makebox[0.7em][r]{$_{\green{5}}\red{4}$} 
			& \makebox[0.7em][r]{$_{\green{9}}1$} 
			& \makebox[0.7em][r]{$_{\green{4}}\red{5}$} 
			& \makebox[0.7em][r]{$_{\green{3}}\red{5}$} 
			& \makebox[0.7em][r]{$_{\green{10}}4$} 
			& \makebox[0.7em][r]{$_{\green{2}}3$} 
			& \makebox[0.7em][r]{$_{\green{1}}1$}
			& \makebox[0.7em][r]{$_{\green{0}}2$}
			& \makebox[0.7em][r]{$_*1$}
			& \makebox[0em][r]{$_*$} \\
		\end{tabular}
		\\[0.5em] 
		{\small $w^{(2)}$} 
	\end{minipage}
	\hfill
	\begin{minipage}{0.45\textwidth}
		\centering
		\vspace{0.3em}
		\fontsize{11}{9}\selectfont  
		\renewcommand{\arraystretch}{1.5}  
		\setlength{\tabcolsep}{0.3em} 
		\begin{tabular}{c|ccccccccccccccc}
			$i$ & 1 & 2 & 3 & 4 & 5 & 6 & 7 & 8 & 9 & 10 &11&12&13 \\
			\hline
			\multirow{2}{*}{$\begin{array}{c} x_i \\ \alpha_i \end{array}$} 
			& \makebox[0.7em][r]{1} & \makebox[0.7em][r]{1} & \makebox[0.7em][r]{1} & \makebox[0.7em][r]{2} & \makebox[0.7em][r]{2} & \makebox[0.7em][r]{2} & \makebox[0.7em][r]{3} & \makebox[0.7em][r]{3} & \makebox[0.7em][r]{4} & \makebox[0.7em][r]{4} & \makebox[0.7em][r]{5} & \makebox[0.7em][r]{5} & \makebox[0.7em][r]{5}\\
			& \makebox[0.7em][r]{$_{\green{7}}\red{3}$}
			& \makebox[0.7em][r]{$_{\green{8}}2$} 
			& \makebox[0.7em][r]{$_{\green{9}}2$} 
			& \makebox[0.7em][r]{$_{\green{6}}\red{4}$} 
			& \makebox[0.7em][r]{$_{\green{5}}\red{5}$} 
			& \makebox[0.7em][r]{$_{\green{4}}\red{5}$} 
			& \makebox[0.7em][r]{$_{\green{3}}\red{5}$} 
			& \makebox[0.7em][r]{$_{\green{10}}1$} 
			& \makebox[0.7em][r]{$_{\green{2}}4$} 
			& \makebox[0.7em][r]{$_{\green{1}}3$}
			& \makebox[0.7em][r]{$_{\green{0}}1$}
			& \makebox[0.7em][r]{$_*2$} 
			& \makebox[0.7em][r]{$_*1$}
			& \makebox[0em][r]{$_*$} \\
		\end{tabular}
		\\[0.5em] 
		{\small $w^{(3)}$} 
	\end{minipage}
\end{table}

\begin{table}[h]
	\centering
	\begin{minipage}{0.5\textwidth}
		\centering
		\vspace{0.3em}
		\fontsize{11}{9}\selectfont  
		\renewcommand{\arraystretch}{1.5}  
		\setlength{\tabcolsep}{0.3em} 
		\begin{tabular}{c|cccccccccccccccc}
			$i$ & 1 & 2 & 3 & 4 & 5 & 6 & 7 & 8 & 9 & 10 &11&12&13&14 \\
			\hline
			\multirow{2}{*}{$\begin{array}{c} x_i \\ \alpha_i \end{array}$} 
			& \makebox[0.7em][r]{1} & \makebox[0.7em][r]{1} & \makebox[0.7em][r]{1} & \makebox[0.7em][r]{2} & \makebox[0.7em][r]{2} & \makebox[0.7em][r]{2} & \makebox[0.7em][r]{3} & \makebox[0.7em][r]{3} & \makebox[0.7em][r]{4} & \makebox[0.7em][r]{4} & \makebox[0.7em][r]{5} & \makebox[0.7em][r]{5} & \makebox[0.7em][r]{5} & \makebox[0.7em][r]{5}\\
			& \makebox[0.7em][r]{$\red{3}$}
			& \makebox[0.7em][r]{$2$} 
			& \makebox[0.7em][r]{$2$} 
			& \makebox[0.7em][r]{$\red{4}$} 
			& \makebox[0.7em][r]{$\red{5}$} 
			& \makebox[0.7em][r]{$\red{5}$} 
			& \makebox[0.7em][r]{$\red{5}$} 
			& \makebox[0.7em][r]{$1$} 
			& \makebox[0.7em][r]{$5$} 
			& \makebox[0.7em][r]{$4$}
			& \makebox[0.7em][r]{$3$}
			& \makebox[0.7em][r]{$1$} 
			& \makebox[0.7em][r]{$2$}
			& \makebox[0.7em][r]{$1$} \\
		\end{tabular}
		\\[0.5em] 
		{\small $w^{(4)}$} 
	\end{minipage}
	\hfill
	\begin{minipage}{0.44\textwidth}
		\centering
		\vspace{0.3em}
		\fontsize{11}{9}\selectfont  
		\renewcommand{\arraystretch}{1.5}  
		\setlength{\tabcolsep}{0.5em} 
		\begin{tabular}{|c|c|c|c|}
			\hline
			$b$ & $\lambda^{(b)}$ & $T^{(b)}$ & $y_b$ \\
			\hline
			0 & $(9,9,5,4)$ & $(5,6,7,4,8,3,9,10,2,1,0)$ & 7 \\
			\hline
			1 & $(9,5,4)$ & $(6,7,8,5,9,4,3,10,2,1,0)$ & 6 \\
			\hline
			2 & $(9,5)$ & $(6,7,8,5,9,4,3,10,2,1,0)$ & 5 \\
			\hline
			3 & $(5)$ & $(7,8,9,6,5,4,3,10,2,1,0)$ & 5 \\
			\hline
		\end{tabular}
		\\[0.5em]
	\end{minipage}
\end{table}

\begin{lemma}\label{lem-y_b}
	For all $0\leq b<k_n-1$, we have $y_{b+1}\leq y_{b}$. 
	\end{lemma}
\begin{proof}
 In order to show that 
$y_{b+1}\leq y_b$, it suffices to show that $$\lambda^{(b+1)}\subseteq \{a^{(b+1)}_{1}, a^{(b+1)}_2, \ldots, a^{(b+1)}_{y_b}\}.$$  Recall that $y_b$ is the greatest integer such that $a^{(b)}_{y_b}\in \lambda^{(b)}$.  By the choice of $y_b$, it is easily seen that
$$
\lambda^{(b+1)} \subseteq \{a^{(b)}_{1}, a^{(b)}_2, \ldots, a^{(b)}_{y_b}\}. 
$$

In order to prove  $y_{b+1}\leq y_b$, it suffices to show that 
\begin{equation}\label{eq-set}
	\{a^{(b)}_{1}, a^{(b)}_2, \ldots, a^{(b)}_{y_b}\}= \{a^{(b+1)}_{1}, a^{(b+1)}_2, \ldots, a^{(b+1)}_{y_b}\}
	\end{equation}
Assume that $w^{(b)}$ has exactly $s$ $g$-gap $h$-level excedance letters.  Recall that    $w^{(b+1)}=\phi_{n, g,h}(w^{(b)}, a^{(b)}_{y_b})$. By (\ref{eq-Excp_r}), we deduce that 
$$
	g\Excp_h(w^{(b+1)})=\left\{ \begin{array}{ll}
		g\Excp_h(w^{(b)})&\, \mathrm{if}\,\,   0\leq a^{(b)}_{y_b}\leq s+\gamma_g  \\
		g\Excp_h(w^{(b)})\cup \{y_b\}&\, \mathrm{otherwise}.
	\end{array}
	\right.
$$
Then  (\ref{eq-set}) follows directly from     the rules specified in $g\den_{h}$-labeling, completing the proof. 
	\end{proof}

\begin{lemma}\label{lem-Phi-den1}
	For all $0\leq b<k_n$, we have $$
	g\den_{h}(w^{(b+1)})=g\den_{h}(w^{(b)})+a^{(b)}_{y_b},
	$$
	and 
	$$
	\psi_{n, g,h}(w^{(b+1)})=(w^{(b)}, a^{(b)}_{y_b}). 
	$$
	
	\end{lemma}
\begin{proof}
	Recall that $w^{(b+1)}=\phi_{n,g,h}(w^{(b)}, a^{(b)}_{y_{b}})$ for all $0\leq b<k_n$.   Let $w^{(b)}=\alpha^{(b)}_{1}\alpha^{(b)}_2\ldots \alpha^{(b)}_{m'+b}$ for all $0\leq b\leq k_n$.
By Lemmas \ref{lem-phi-den1} and \ref{lem-psi-den1}, 	in order to prove $
	g\den_{h}(w^{(b+1)})=g\den_{h}(w^{(b)})+a^{(b)}_{y_b}$ and $
	\psi_{n,g,h}(w^{(b+1)})=(w^{(b)}, a^{(b)}_{y_b})
	$, it suffices to show that  there does not exist  any $n's$ occurring to the left of $\alpha^{(b)}_{y_b}$ in $w^{(b)}$ for all  $0\leq b<k_n$.    We proceed to prove the assertion  by induction on $b$.  Clearly, the assertion holds for $b=0$.  Assume that the assertion also holds for $b-1$, that is,  there does not exist any $n's$ occurring to the left of $\alpha^{(b-1)}_{y_{b-1}}$ in $w^{(b-1)}$.  
	 Recall that  the map $\phi_{n,g,h}$ keeps the letters occurring to the left of $\alpha^{(b-1)}_{y_{b-1}}$ in place. 
	  This implies that $\alpha^{(b)}_{y_{b-1}}=n$ is the leftmost $n$ in $w^{(b)}$. 
	    Lemma \ref{lem-y_b} tells that $y_{b}\leq y_{b-1}$.  This yields  that there does not exists any $n's$ occurring to the left of $\alpha^{(b)}_{y_{b}}$ in $w^{(b)}$ by induction hypothesis, completing the proof.

	\end{proof}

\begin{lemma}\label{lem-Phi-den2}
 If $g\exc_\ell(w)=s$ and $g\exc_\ell(w^{(k_n)})=t$, then we have
$$
	\lambda_{t-s}\geq t+\delta_\ell+\gamma_g \geq \lambda_{t-s+1}.
 $$
 	
	\end{lemma}
  \begin{proof}
  	Recall that  $w^{(b)}=\phi_{n,g,h}(w^{(b-1)}, a^{(b-1)}_{y_{b-1}})$. Then by Lemma \ref{lem-phi-den3},  we     have either  $g\exc_\ell(w^{(b)})=g\exc_\ell(w^{(b-1)})$ or $g\exc_\ell(w^{(b)})=g\exc_\ell(w^{(b-1)})+1$. 
  	Let $t-s=k$. 
  	 Since $g\exc_\ell(w^{(k_n)})-g\exc_\ell(w)=k$,   there are exactly $k$  integers $b$ such that $g\exc_\ell(w^{(b+1)})=g\exc_\ell(w^{(b)})+1$.    
  	Suppose that such integers are given by  $i_1, i_2, \ldots, i_{k}$ with $i_1<i_2<\cdots<i_{k}$.  
  	  Again by Lemma \ref{lem-phi-den3}, we deduce that  
  	  $$a^{(b)}_{y_{b}}\leq g\exc_\ell(w^{(b)})+\gamma_g+\delta_\ell\leq g\exc_\ell(w^{(k_n)})+\gamma_g+\delta_\ell=t+\gamma_g+\delta_\ell$$ for all 
  	  $b\neq i_j $ and $1\leq j\leq k$.  In order to show that 	$\lambda_{k}\geq t+\delta_\ell+\gamma_g \geq \lambda_{k+1}$, it suffices to show that 	$a^{(z)}_{y_{z}}= t+\gamma_g+\delta_\ell $  when  $z=i_j$ for  any $1\leq j\leq k$.

  	   By the choice of $i_k$,  we have $g\exc_\ell(w^{(i_k+1)})=g\exc_\ell(w^{(i_k)})+1$  and $g\exc_\ell(w^{(i_k+1)})=t$.  
  	Again by  Lemma \ref{lem-phi-den3}, we have
  	$$
  	a^{(i_{k})}_{y_{{i_{k}}}}\geq   g\exc_\ell(w^{(i_{k})})+\gamma_g+\delta_\ell +1=  t+\gamma_g+\delta_\ell. 
  	$$
   In the following, we aim to show that
  		$a^{(i_{j+1})}_{y_{{i_{j+1}}}}\leq a^{(i_{j})}_{y_{{i_{j}}}}$ for all $1\leq j <k$. 
  	Since $g\exc_\ell(w^{(i_{j}+1)})=g\exc_\ell(w^{(i_j)})+1$, by Lemma \ref{lem-phi-den3}, we have $a^{(i_{j})}_{y_{{i_{j}}}}> g\exc_\ell(w^{(i_{j})})+\gamma_g+\delta_\ell $. 
  	Assume that $w^{(i_{j})}$ has exactly $z$ $g$-gap $h$-level excedance letters. 
  	Since $h\leq g+\ell$,  one can easily check that $z\leq g\exc_\ell(w^{(i_j)})+\delta_\ell$. Then $a^{(i_{j})}_{y_{{i_{j}}}}> g\exc_\ell(w^{(i_{j})})+\gamma_g+\delta_\ell $ would imply that $\alpha^{(i_j)}_{y_{i_j}}$  is a non-$g$-gap-$h$-level excedance letter in $w^{(i_j)}$ by the rules specified in the $g\den_h$-labeling.    Again by the rules specified in the $g\den_{h}$-labeling, we have
  	\begin{equation}\label{eq-label}
  	a^{(i_j)}_{y_{i_j}}=\max\{a^{(i_j)}_1,  a^{(i_j)}_2, \ldots, a^{(i_j)}_{y_{i_j}}\}. 
  	\end{equation}
  	Note that $y_{i_j}$ is the greatest integer    such that $a^{(i_j)}_{y_{i_j}}\in \lambda^{(i_j)}$.  Since $a^{(i_{j+1})}_{y_{i_{j+1}}}\in \lambda^{(i_j)}$,   we have $a^{(i_{j+1})}_{y_{i_{j+1}}}\in \{a^{(i_j)}_1,  a^{(i_j)}_2, \ldots, a^{(i_j)}_{y_{i_j}}\}.$   This combined with (\ref{eq-label}) yields that 
  	$a^{(i_{j+1})}_{y_{i_{j+1}}}\leq a^{(i_j)}_{y_{i_j}}$ as desired, 
  	completing the proof. 
   
  	\end{proof}

  \noindent{\bf Proof of Theorem \ref{th-Phi-den}.}  Properties (\ref{eq-Phiden1}) and (\ref{eq-Phiden2}) follow  directly from  Lemmas \ref{lem-Phi-den1} and \ref{lem-Phi-den2}. By cardinality reasons, in order to show that $\Phi^{\den}_{g,h}$ is a bijection, it suffices to show that $\Phi^{\den}_{g,h}$ is an injection. 
  Let $(w, \lambda)\in \mathfrak{S}_{M'}\times \mathcal{P}(k_n, m')$ and let $u=\Phi^{\den}_{g,h}(w, \lambda)$. 
  We retain  all the notations  from  the definition of $\Phi^{\den}_{g,h}$.  By Lemma \ref{lem-Phi-den1}, for all $1\leq b\leq k_n$, we can recover 
  $(w^{(b-1)}, a^{(b-1)}_{y_{b-1}})$ by  applying the map $\psi_{n,g,h}$ to $w^{(b)}$. This establishes the injectivity of  $\Phi^{\den}_{g,h}$  and hence it  is a bijection, completing the proof. \qed

   \subsection{Proof of Theorem \ref{th-Phi-maj}} This subsection is devoted to constructing the bijection $\Phi^{\maj}_{g,\ell}$.  To this end, we introduce a new labeling scheme.  
  Throughout this section, we   assume that  $n\geq g+\ell$ and   $N=M'\cup\{n^a\}$ with $a\geq 0$. 
   \begin{definition}
   	Given $w=\alpha_1\alpha_2\ldots \alpha_{m'+a}\in \mathfrak{S}_{N}$, 
   	let
   	$$
   	S(w)=\{j\mid \alpha_{j-1}\geq \alpha_j+g, \alpha_j\geq \ell\}\cup \{j\mid  \alpha_j<\ell\,\, \mbox{or}\,\, n>\alpha_j\geq n-g+1\}. 
   	$$
   	The {\bf  $g\maj_\ell$-labeling} of $w$  is obtained as follows.
   	\begin{itemize}
   		\item Star the space before each $n$.
   		\item Label the  space after $\alpha_{m'+a}$ by $0$.
   		\item Label the   space  before each  $\alpha_j$ with $j\in  S(w)$ from right to left with $1,2,\ldots, |S(w)|$. 
   		\item  Label the remaining unstarred   spaces   from left to right with $|S(w)|+1, \ldots, m' $.
   	\end{itemize}
   \end{definition}
   It is apparent that $|S(w)|=g\des_{\ell}(w)+\delta_\ell+\gamma_g$. 
   For example, let  $N=\{1^3, 2^3, 3, 4, 5^2, 6^3, 7, 8\}$.  It is apparent that the 
   $2\maj_2$-labeling of $w=312117248656625\in \mathfrak{S}_{N}$ is given by
   $$
   \begin{array}{llllllllllllllll} 
   _8 3&_{{\red 7}}1&_92&_{\red{6}}1&_{\red{5}}1&_{\red{4}}7&_{\red{3}}2&_{10}4&*8&_{\red{2}}6&_{11}5&_{12}6&_{13}6&_{\red{1}}2&_{14}5&_0\\
   \end{array}
   $$
    where the   label  of the space before each $\alpha_j$  with  $j\in S(w)$ is in red.

   Now we are in the position to construct the map $\Phi^{\maj}_{g, \ell}$. 
   
    \begin{framed}
   	\begin{center}
   		{\bf The map $\Phi^{\maj}_{g,\ell}:\mathfrak{S}_{M'}\times \mathcal{P}(k_n, m') \longrightarrow \mathfrak{S}_{M}$ }
   	\end{center}
   Given  $(w, \lambda)\in \mathfrak{S}_{M'}\times \mathcal{P}(k_n, m') $, let $\lambda^{(0)}=\lambda=(\lambda_1, \lambda_2, \ldots, \lambda_{k_n})$ and  $w^{(0)}=w$. Set $b=0$ and carry out the following procedure. 
   	
   	(A) Define 
   	$$T^{(b)}=(a^{(b)}_{1}, a^{(b)}_2, \ldots,  a^{(b)}_{m'+1})$$ where $a^{(b)}_{i}$  is label of the $i$-th  unstarred space of  $w^{(b)}$ (counting from left to right) under the $g\maj_\ell$-labeling of $w^{(b)}$ for all $1\leq i\leq m'+1$. 
   	Choose $y_b$ to the greatest  integer such that $a^{(b)}_{y_b}\in \lambda^{(b)}$ and let  
 $w^{(b+1)} $  be the permutation obtained from $w^{(b)}$ by inserting an $n$ into the $y_{b}$-th unstarred space of $w^{(b)}$. Set $\lambda^{(b+1)}=\lambda^{(b)}\setminus \{ a^{(b)}_{y_b}\}$.

   	(B)  Replace $b$ by $b+1$. If $b=k_n$, then we stop. Otherwise, we go back to (A). 
   	
   	Define $\Phi^{\maj}_{g,h}(w, \lambda)=w^{(k_n)}$.

   \end{framed}

Take $M=\{1^3, 2,3,4^2, 5^2, 6, 7^4\}$, $w=4151652413$ and $\lambda=(9,9,6,3)$ for example. Then we can obtain $\Phi^{\maj}_{2,3}(w, \lambda)=w^{(4)}= 41571675727413$ as follows.  
\begin{table}[h]
	\centering
  	\begin{minipage}[t]{0.95\textwidth}  
  	\renewcommand{\arraystretch}{1.6}
  	\setlength{\tabcolsep}{3pt}  
  	\fontsize{9}{11}\selectfont  
  	
  	\begin{tabular}{|c|>{\centering\arraybackslash}m{5.5cm}|>{\centering\arraybackslash}m{5.5cm}|>{\centering\arraybackslash}m{2.8cm}|}
  		\hline
  		$b$ & $w^{(b)}$ & $T^{(b)}$ & $\lambda^{(b)}$ \\
  		\hline
  		0 & 
  		$_6 4_{\color{red}5}1_7 5_{\color{red}4}1_{\color{red}3} 6_{8} 5_{\color{red}2} 2_9 4_{\color{red}1} 1_{10} 3_0$ & 
  		$(6,5,7,4,3,8,2,9,1,10,0)$ & 
  		$(9,9,6,3)$ \\
  		\hline
  		1 & 
  		$_7 4_{\color{red}6}1_{8}5_{\color{red}5}1_{\color{red}4} 6_{9} 5_{\color{red}3} 2_* 7_{\color{red}2} 4_{\color{red}1} 1_{10} 3_0$ &  
  		$(7,6,8,5,4,9,3,2,1,10,0)$ & 
  		$(9,6,3)$ \\
  		\hline
  		2 & 
  		$_7 4_{\color{red}6}1_{8}5_{\color{red}5}1_{\color{red}4} 6_{9} 5_* 7_{\color{red}3} 2_* 7_{\color{red}2} 4_{\color{red}1} 1_{10} 3_0$&  
  		$(7,6,8,5,4,9,3,2,1,10,0)$ & 
  		$(9,6)$ \\
  		\hline
  		3 & 
  		$_8 4_{\color{red}7}1_{9}5_{\color{red}6}1_{\color{red}5} 6_* 7_{\color{red}4} 5_* 7_{\color{red}3} 2_* 7_{\color{red}2} 4_{\color{red}1} 1_{10} 3_0$&  
  		$(8,7,9,6,5,4,3,2,1,10,0)$ & 
  		$(6)$ \\
  		\hline
  	\end{tabular}
  	\\[0.5em] 
  \end{minipage}
  
   \end{table}

  With a careful examination of 
  the construction of $w^{(b)}$ and the rules specified in the $g\maj_\ell$-labeling of $w^{(b)}$, we obtain the following lemma.

  \begin{lemma}\label{lem-label-maj}
  	For all $0\leq b<k_n-1$, we have 
  	$$
  	\{a^{(b+1)}_1, a^{(b+1)}_2, \ldots, a^{(b+1)}_{y_{b}}\}= 	\{a^{(b)}_1, a^{(b)}_2, \ldots, a^{(b)}_{y_{{b}}}\}. 
  	$$
  	Moreover, we have 
  	$$
  	a^{(b)}_{y_b}=\min\{a^{(b)}_1, a^{(b)}_2, \ldots, a^{(b)}_{y_{{b}}}\}
  	$$
 when $a^{(b)}_{y_b}\leq |S(w^{(b)})|$, and 
 	$$
 a^{(b)}_{y_b}=\max\{a^{(b)}_1, a^{(b)}_2, \ldots, a^{(b)}_{y_{{b}}}\}
 $$
  otherwise.
  	\end{lemma}

  \begin{lemma}\label{lem-y_b-maj}
  	For all $0\leq b<k_n-1$, we have $y_{b+1}\leq y_{b}$. 
  \end{lemma}
  \begin{proof}
  	In order to show that 
  	$y_{b+1}\leq y_b$, it suffices to show that $$\lambda^{(b+1)}\subseteq \{a^{(b+1)}_{1}, a^{(b+1)}_2, \ldots, a^{(b+1)}_{y_b}\}.$$      By the choice of $y_b$, it is easily seen that
  	$$
  	\lambda^{(b+1)} \subseteq \{a^{(b)}_{1}, a^{(b)}_2, \ldots, a^{(b)}_{y_b}\}. 
  	$$
  By	Lemma \ref{lem-label-maj}, it follows that
  $$\lambda^{(b+1)}\subseteq \{a^{(b)}_{1}, a^{(b)}_2, \ldots, a^{(b)}_{y_b}\}=\{a^{(b+1)}_{1}, a^{(b+1)}_2, \ldots, a^{(b+1)}_{y_b}\} $$ as desired, completing the proof.  
  \end{proof}

   \begin{lemma}\label{lem-Phi-maj1}
  	For all $0\leq  b<k_n$, we have
  	\begin{equation}\label{eq-maj0}
  		g\maj_\ell(w^{(b+1)}) = 	 g\maj_\ell (w^{(b)})+a^{(b)}_{y_b}
  	\end{equation}
  \end{lemma}
  \begin{proof}
  	 From   the construction of $w^{(b+1)}$ and the rules specified in the $g\maj_\ell$-labeling of $w^{(b)}$,  it is straightforward to verify   that 
  	 $$
  	 g\maj_{\ell}(w^{(b+1)})=\left\{ \begin{array}{ll}
  	 g\maj_{\ell}(w^{(b)})+a^{(b)}_{y_b}&\, \mathrm{if}\,\,   0\leq a^{(b)}_{y_b}\leq  |S(w^{(b)})|  \\
  	 [1em]
  	 g\maj_{\ell}(w^{(b)})+a^{(b)}_{y_b}+z&\, \mathrm{otherwise},
  	 \end{array}
  	 \right.
  	 $$
  	 where $z$ denotes   the number of occurrences of $n$  to the left of  the $y_b$-th unstarred space of $w^{(b)}$. 
  	To prove  (\ref{eq-maj0}), it suffices to show that $z=0$ whenever $a^{(b)}_{y_b}>  |S(w^{(b)})| $.

  	Now we assume that  $a^{(b)}_{y_b}>  |S(w^{(b)})| $.   Let $w^{(b)}=\alpha^{(b)}_1\alpha^{(b)}_2\ldots \alpha^{(b)}_{m'+b}$.  Again by the rules  in the definition of the  $g\maj_\ell$-labeling,    the $y_b$-th unstarred space must lie  immediately before  some $\alpha^{(b)}_j$ with $j\notin S(w^{(b)})$.       By Lemma  \ref{lem-y_b-maj}, 
  	we have $y_{b}\leq y_{b-1}$.   Recall that $w^{(b)}$ is obtained from $w^{(b-1)}$ by inserting an $n$ into the $y_{b-1}$-th unstarred space of $w^{(b-1)}$.  This implies that the $y_{b-1}$-th  unstarred space   must lie immediately before some $\alpha^{(b)}_j$ with $j\in S(w^{(b)})$. Hence, we have $y_{b}<y_{b-1}$.  Consequently,  there is  no occurrence of $n$   preceding  the $y_b$-th unstarred space of $w^{(b)}$, which yields that  $z=0$ and  completes  the proof. 
  	\end{proof}
  \begin{lemma}\label{lem-Phi-maj3}
  	If $g\des_\ell(w)=s$ and $g\des_\ell(w^{(k_n)})=t$, then we have
  	$$
  	\lambda_{t-s}\geq t+\delta_\ell+\gamma_g \geq \lambda_{t-s+1}.
  	$$
  	
  \end{lemma}
  \begin{proof}
  By the construction of $w^{(b)}$,  we     have either  $g\des_\ell(w^{(b)})=g\des_\ell(w^{(b-1)})$ or $g\des_\ell(w^{(b)})=g\des_\ell(w^{(b-1)})+1$ for all $1\leq b\leq k_n$.  
  	Let $t-s=k$. 
  	Since $g\des_\ell(w^{(k_n)})-g\des_\ell(w)=k$,   there are exactly $k$  integers $b$ such that $g\des_\ell(w^{(b+1)})=g\des_\ell(w^{(b)})+1$.    
  	Suppose that such integers are given by  $i_1, i_2, \ldots, i_{k}$ with $i_1<i_2<\cdots<i_{k}$.  
  	By the rules specified in the $g\maj_\ell$-labeling, we have
  	$$a^{(b)}_{y_{b}}\leq |S(w^{(b)})|  =g\des_\ell(w^{(b)})+\gamma_g+\delta_\ell\leq g\des_\ell(w^{(k_n)})+\gamma_g+\delta_\ell=t+\gamma_g+\delta_\ell$$ for all 
  	$b\neq i_j $ and $1\leq j\leq k$.  In order to show that 	$\lambda_{k}\geq t+\delta_\ell+\gamma_g \geq \lambda_{k+1}$, it suffices to show that 	$a^{(z)}_{y_{z}}\geq t+\gamma_g+\delta_\ell $  when  $z=i_j$ for  any $1\leq j\leq k$.

  	By the choice of $i_k$,  we have $g\des_\ell(w^{(i_k+1)})=g\des_\ell(w^{(i_k)})+1$  and $g\des_\ell(w^{(i_k+1)})=t$.  
  	Again by the rules specified in the $g\maj_\ell$-labeling,  we have
  	$$
  	a^{(i_{k})}_{y_{{i_{k}}}}\geq  |S(w^{(i_k)})|+1 = 1+g\des_\ell(w^{(i_{k})})+\gamma_g+\delta_\ell = t+\gamma_g+\delta_\ell. 
  	$$
  	In the following, we aim to show that
  	$a^{(i_{j+1})}_{y_{{i_{j+1}}}}\leq a^{(i_{j})}_{y_{{i_{j}}}}$ for all $1\leq j <k$. 
  	Since $g\des_\ell(w^{(i_{j}+1)})=g\des_\ell(w^{(i_j)})+1$, again by the rules specified in the $g\maj_\ell$-labeling, we deduce that  $$a^{(i_{j})}_{y_{{i_{j}}}}\geq |S(w^{(i_j)})|+1= g\des_\ell(w^{(i_{j})})+\gamma_g+\delta_\ell+1$$  and thus we have
  	\begin{equation}\label{eq-label-maj}
  		a^{(i_j)}_{y_{i_j}}=\max\{a^{(i_j)}_1,  a^{(i_j)}_2, \ldots, a^{(i_j)}_{y_{i_j}}\}. 
  	\end{equation}
  by Lemma \ref{lem-label-maj}. 
  	Note that $y_{i_j}$ is the greatest integer    such that $a^{(i_j)}_{y_{i_j}}\in \lambda^{(i_j)}$.  Since $a^{(i_{j+1})}_{y_{i_{j+1}}}\in \lambda^{(i_j)}$,   we have $a^{(i_{j+1})}_{y_{i_{j+1}}}\in \{a^{(i_j)}_1,  a^{(i_j)}_2, \ldots, a^{(i_j)}_{y_{i_j}}\}.$   This combined with (\ref{eq-label-maj}) yields that 
  	$a^{(i_{j+1})}_{y_{i_{j+1}}}\leq a^{(i_j)}_{y_{i_j}}$ as desired,  
  	completing the proof. 
  	
  \end{proof}
  
  Now we are in the position to finish the proof of Theorem   \ref{th-Phi-maj}.
  
  \noindent{\bf Proof of Theorem  \ref{th-Phi-maj}.}
  Properties~\eqref{eq-Phi1} and ~\eqref{eq-Phi2}  follow  directly from Lemmas \ref{lem-Phi-maj1}  and  \ref{lem-Phi-maj3}.  By cardinality reasons, it remains to show that $\Phi^{\maj}_{g,\ell}$ is an injection.      Here we retain  all the notations  in the definition of $\Phi^{\maj}_{g,h}$.   By the construction of $\Phi^{\maj}_{g,h}$, to recover $(w, \lambda)$, it suffices to show that  $(w^{(b)}, a^{(b)}_{y_b})$ can be retrieved from $w^{(b+1)}$ for all $0\leq b<k_n$.  Lemmas \ref{lem-y_b-maj} and  \ref{lem-Phi-maj1}  ensure that   we can recover $w^{(b)}$ from $w^{(b+1)}$  by deleting  the leftmost $n$ and recover $a^{(b)}_{y_b}$ by letting $a^{(b)}_{y_b}=g\maj_\ell(w^{(b+1)})-g\maj_{\ell}(w^{(b)})$. This concludes the proof.   \qed

   \section*{Acknowledgments}  
   The work  was supported by
   the National Natural
   Science Foundation of China grants 12471318 and 12071440.


\end{document}